\documentclass[11pt,reqno]{amsart}
\usepackage[utf8]{inputenc}
\usepackage[T1]{fontenc}
\usepackage[english]{babel}
\usepackage{csquotes}

\usepackage[margin=3cm]{geometry}

\usepackage{xcolor}
\usepackage[colorlinks=true, linkcolor=black, citecolor=black]{hyperref} 
\usepackage[nameinlink]{cleveref}

\usepackage[onehalfspacing]{setspace}

\numberwithin{equation}{section} 

\usepackage{paralist}
\usepackage{enumitem}
\setlist[enumerate,1]{label = (\alph*),leftmargin=*, topsep=1mm, itemsep=1mm}
\setlist[enumerate,2]{label = (\roman*),leftmargin=*, topsep=1mm, itemsep=1mm}
\setlist[itemize,1]{leftmargin=*, topsep=1mm, itemsep=1mm}
\usepackage{times}
\usepackage{txfonts}
\usepackage{dsfont}
\renewcommand{\mathbb}{\mathds}

\usepackage{amsmath,amssymb}
\usepackage{mathdots}
\usepackage{extpfeil}
\usepackage{extarrows}
\usepackage{array}

\usepackage{tikz}
\usetikzlibrary{cd,arrows,decorations.markings}

\tikzset{
	marrow/.style={decoration={markings,mark=at position 0.75 with {\arrow{#1}}}, postaction=decorate}
}

\tikzcdset{
	arrow style=tikz,
	diagrams={>={Straight Barb[scale=0.9, width=5pt]}} }

\tikzset{epi/.code={\pgfsetarrowsend{Computer Modern Rightarrow[width=5pt, length=3pt] Computer Modern Rightarrow[width=5pt, length=3pt]}}}

\tikzset{epi_mini/.code={\pgfsetarrowsend{Computer Modern Rightarrow[width=5pt, length=3pt, scale=0.85] Computer Modern Rightarrow[width=5pt, length=3pt, scale=0.85]}}}

\usepackage[backend=biber,style=alphabetic,maxbibnames=4,doi=false,isbn=false,url=false]{biblatex}
\bibliography{catmatfac.bib}

\theoremstyle{plain}
\newtheorem{thm}{Theorem}[section]
\newtheorem{prp}[thm]{Proposition}
\newtheorem{cor}[thm]{Corollary}
\newtheorem{lem}[thm]{Lemma}

\newtheorem{thmA}{Theorem}

\theoremstyle{definition}
\newtheorem{dfn}[thm]{Definition}
\newtheorem{ntn}[thm]{Notation}
\newtheorem{con}[thm]{Construction}
\newtheorem{asn}[thm]{Assumption}

\theoremstyle{remark}
\newtheorem{rmk}[thm]{Remark}

\Crefname{thm}{Theorem}{Theorems}
\Crefname{thmA}{Theorem}{Theorems}
\Crefname{prp}{Proposition}{Propositions}
\Crefname{cor}{Corollary}{Corollaries}
\Crefname{lem}{Lemma}{Lemmas}
\Crefname{cnj}{Conjecture}{Conjectures}
\Crefname{dfn}{Definition}{Definitions}
\Crefname{ntn}{Notation}{Notations}
\Crefname{con}{Construction}{Constructions}
\Crefname{asn}{Assumption}{Assumptions}
\Crefname{cnv}{Convention}{Conventions}
\Crefname{rmk}{Remark}{Remarks}
\Crefname{exa}{Example}{Examples}

\newcommand{\NN}{\mathbb{N}}

\newcommand{\A}{\mathcal{A}}
\newcommand{\B}{\mathcal{B}}
\newcommand{\C}{\mathcal{C}}

\newcommand{\E}{\mathcal{E}}
\newcommand{\F}{\mathcal{F}}
\newcommand{\cI}{\mathcal{I}}

\renewcommand{\L}{\mathcal{L}}
\newcommand{\N}{\mathcal{N}}
\newcommand{\cS}{\mathcal{S}}
\newcommand{\T}{\mathcal{T}}

\newcommand{\mMor}{\mm\Mor}
\newcommand{\smMor}{\sm \Mor}

\newcommand{\ul}{\underline}
\newcommand{\ol}{\overline}

\newcommand{\sm}[1]{{#1}^\textup{sm}}
\newcommand{\mm}[1]{{#1}^\textup{m}}

\newcommand{\epic}[2]{\begin{tikzcd}[cramped, sep=small, ampersand replacement=\&] #1 \ar[r, two heads] \& #2 \end{tikzcd}}
\newcommand{\monic}[2]{\begin{tikzcd}[cramped, sep=small, ampersand replacement=\&] #1 \ar[r, tail] \& #2 \end{tikzcd}}
\newcommand{\inj}[2]{\begin{tikzcd}[cramped, sep=small, ampersand replacement=\&] #1 \ar[r, hook] \& #2 \end{tikzcd}}

\DeclareMathOperator{\id}{id}
\DeclareMathOperator{\cok}{cok}
\DeclareMathOperator{\Hom}{Hom}
\DeclareMathOperator{\Ext}{Ext}
\DeclareMathOperator{\Mor}{Mor}
\DeclareMathOperator{\Proj}{Proj}
\DeclareMathOperator{\Inj}{Inj}
\DeclareMathOperator{\gdim}{gdim}
\DeclareMathOperator{\Fac}{Fac}
\DeclareMathOperator{\MF}{MF}
\DeclareMathOperator{\MCM}{MCM}
\DeclareMathOperator{\GProj}{Gproj}

\title[Categorical matrix factorizations and monomorphism categories]{Categorical matrix factorizations\\and monomorphism categories}

\author[J.~Frank]{Jonas Frank}
\address{\linebreak
	Jonas Frank\\
	Department of Mathematics\\
	RPTU University Kaiserslautern-Landau\\ 
	67663 Kaiserslautern\\
	Germany
}
\email{\href{jfrank@rptu.de}{jfrank@rptu.de}}

\author[M.~Schulze]{Mathias Schulze}
\address{\linebreak
	Mathias Schulze\\
	Department of Mathematics\\
	RPTU University Kaiserslautern-Landau\\ 
	67663 Kaiserslautern\\
	Germany
}
\email{\href{mschulze@rptu.de}{mschulze@rptu.de}}

\subjclass[2020]{Primary 18G50, 18G65; Secondary 18G80, 16E65}

\keywords{Matrix factorization, monomorphism category, triangulated category, Cohen-Macaulay, Gorenstein}

\begin{document}
	
%%%%%%%%%%%%%%%%%%%%%%%%%%%%%%%%%%%%%%%%%%%%%%%%%%%%%%%%%%%%%%%%%%%%%%%%%%%%%%%%

\begin{abstract}
	This article generalizes the correspondence between matrix factorizations and maximal Cohen-Macaulay modules over hypersurface rings due to Eisenbud and Yoshino. We consider factorizations with several factors in a purely categorical context, extending results of Sun and Zhang for Gorenstein projective module factorizations. Our formulation relies on a notion of hypersurface category and replaces Gorenstein projectives by objects of general Frobenius exact subcategories. We show that factorizations over such categories form again a Frobenius category. Our main result is then a triangle equivalence between the stable category of factorizations and that of chains of monomorphisms.
\end{abstract}

\maketitle

\tableofcontents

\section{Introduction}

Eisenbud {\cite{Eis80}} showed that any minimal free resolution over a hypersurface ring $R=S/\langle f \rangle$ becomes 2-periodic after $\dim(S)$ many steps. He establishes a correspondence between the isomorphism classes of such minimal 2-periodic resolutions, maximal Cohen-Macaulay (MCM) modules over $R$ without free summands, and equivalence classes of reduced matrix factorizations of $f$: A \emph{matrix factorization} of a non-zero divisor $f$ of a regular local ring $S$ is a pair $(\alpha, \beta)$ of homomorphisms $\alpha, \beta \colon S^m \to S^m$ of free $S$-modules of the same rank $m$ such that $\alpha \circ \beta = \beta \circ \alpha = f \cdot \id_{S^m}$. It gives rise to a 2-periodic free resolution \[\begin{tikzcd} \cdots \ar[r, "\ol \beta"] & R^m \ar[r, "\ol \alpha"] & R^m \ar[r, "\ol \beta"] & R^m \ar[r, "\ol \alpha"] & R^m \ar[r] & \cok(\alpha) \ar[r] & 0 \end{tikzcd}\]
of the cokernel of $\alpha$, which is an MCM $R$-module.\\
Yoshino {\cite{Yos90}} gave a first categorical formulation of Eisenbud's matrix factorization theorem: The cokernel functor from the category $\MF_S(f)$ of matrix factorizations induces respective equivalences \[\MF_S(f)/\langle (\id_S, f) \rangle \simeq \MCM(R), \; (\alpha, \beta) \mapsto \cok(\alpha),\] to the category of MCM $R$-modules and $\MF_S(f)/\langle (\id_S, f), (f, \id_S) \rangle \simeq \ul{\MCM}(R)$ to its stable category.\\
Several authors {\cite{BHS88,HUB91,Tri21,LT23}} studied matrix factorizations with $n \geq 2$ factors. Tribone {\cite{Tri21}} showed that the $n$-fold matrix factorizations of $f$ form a Frobenius category $\MF^n_S(f)$ and described the projective-injectives. As a result, its stable category $\ul{\MF}^n_S(f)$ is triangulated due to Happel {\cite{Hap88}}. Yoshino's second equivalence then becomes a triangle equivalence \[\ul{\MF}_S(f) \simeq \ul{\MCM}(R).\]
Recent generalizations of this equivalence rely on the fact that Gorenstein projective (aka totally reflexive) modules over Gorenstein rings are MCM modules, which, in turn, are free over regular local rings: Chen {\cite{Che25}} for $n=2$ and Sun and Zhang {\cite{SZ26}} for general $n$ consider a regular, normal element $\omega$ of a left Noetherian ring $A$. This gives rise to an autoequivalence $\tau$ of the category of left $A$-modules, defined by $\tau(a) \omega = \omega a$ for $a \in A$. They consider an instance ${}^0{\Fac}^{\GProj(A)}_{n+1}(\omega)$ of the generalized factorization category, introduced by Bergh and Jorgensen {\cite{BJ23}}, where $\GProj(A)$ is the category of Gorenstein projective left $A$-modules. It consists of sequences
	\[\begin{tikzcd}
	X^0 \ar[r, hook, "\alpha^0"] & X^1 \ar[r, hook, "\alpha^1"] & \cdots \ar[r, hook, "\alpha^{n-1}"] & X^n \ar[r, hook, "\alpha^n"] & {}^\tau{\left(X^0\right)}
\end{tikzcd}\]
of left $A$-module monomorphisms composing to $\alpha^n \cdots \alpha^0 = \omega$, where $X^k \in \GProj(A)$ for $k \in \{0, \dots, n-1\}$, $X^n \in \Proj(A)$, and ${}^\tau{\left(X^0\right)}$ has a twisted left $A$-module structure. The role of $\MCM(R)$ is played by the category $\mMor_{n-1}(\GProj(A/\langle \omega \rangle))$ of chains of $n-1$ monomorphism of Gorenstein projective left $A/\langle \omega \rangle$-modules. The authors establish a triangle equivalence 
\[{}^0{\ul{\Fac}}^{\GProj(A)}_{n+1}(\omega) \simeq \ul \mMor_{n-1}(\GProj(A/\langle \omega \rangle))\]
of stable categories, induced by a generalized \emph{cokernel functor}.

\smallskip

In this article, the Eisenbud-Yoshino theorem is phrased in purely categorical terms, isolating the essential hypotheses used for its proof. To this end, we replace the category of left $A$-modules by a general exact category $\A$. We introduce the \emph{hypersurface category} $\A/\omega$ with respect to a \textit{twisted homothety} $(\tau, \omega)$ on $\A$ to mimic the category of left $A/\langle \omega \rangle$-modules. It is given by an additive automorphism $\tau$ of $\A$, which preserves and reflects short exact sequences, and a natural transformation $\omega \colon \id_\A \to \tau$ such that $\omega \tau = \tau \omega$. The roles of $\GProj(A)$ and $\GProj(A/\langle \omega \rangle)$ are played by respective fully exact subcategories $\E=\tau \E$ of $\A$ and $\E_\omega$ of $\A/\omega$, subject to a list of conditions:

\begin{asn} \label{asn} \
	\begin{enumerate}
		\item \label{asn-coker} If \begin{tikzcd}[cramped, sep=small]
			X \ar[r, tail] & Y \ar[r, two heads] & Z
		\end{tikzcd} is any short exact sequence in $\A$, then $X,Y \in \E$ and $Z \in \A/\omega$ implies $Z \in \E_\omega$. 
		
		\item \label{asn-ker} If \begin{tikzcd}[cramped, sep=small]
			X \ar[r, tail] & Y \ar[r, two heads] & Z
		\end{tikzcd} is any short exact sequence in $\A$, then $X \in \E$ follows from $Y \in \E$ and $Z \in \E_\omega$.
		
		\item \label{asn-epic} Every $Z \in \E_\omega$ admits an admissible epic $\epic Y Z$ in $\A$ with $Y \in \E$.
	\end{enumerate}
\end{asn}

While \Cref{asn}.\ref{asn-epic} holds trivially in the module case, \ref{asn-coker} and \ref{asn-ker} represent a consequence of the \emph{change of rings formula} for Gorenstein dimensions. The \emph{Gorenstein dimension} $\gdim(M)$ of a module $M$ is the minimal length of a resolution of $M$ by Gorenstein projectives.

\begin{thm}[Change of Rings, {\cite[Thm.~2.2.8]{Chr00}}] \label{thm: CoR}
	Let $S$ be a commutative ring and $R = S/\langle f_1, \dots, f_k \rangle$, where $f_1, \dots, f_k$ is an $S$-regular sequence. Then $\gdim_S(M) = \gdim_{R}(M)+k$ for any finitely generated $R$-module $M$. 
\end{thm}

To make the link with our assumptions, consider a short exact sequence
\begin{center}
	\begin{tikzcd}
		0 \ar[r] & K \ar[r] & M \ar[r] & N \ar[r] & 0
	\end{tikzcd}
\end{center}
of $S$-modules where $M$ is Gorenstein projective over $S$ and $N$ an $R$-module. Then $K$ is Gorenstein projective over $S$ if and only if $\gdim_S(N)=1$, see {\cite[Thm.~1.2.7]{Chr00}}. By \Cref{thm: CoR}, this latter condition is equivalent to $\gdim_R(N)=0$, or to $N$ being Gorenstein projective over $R$.\\
The \emph{monomorphism category} $\mMor_{l}(\A)$ consists of chains 
\[
\begin{tikzcd}
	X = (X, \alpha) \colon \; X^0 \ar[r, tail, "\alpha^0"] & X^1 \ar[r, tail, "\alpha^1"] & \cdots \ar[r, tail, "\alpha^{l-1}"] & X^l
\end{tikzcd}
\]
of $l$ monomorphisms in $\A$. In our setting, an \emph{$\E$-factorization} of $\omega$ with $l+1$ factors is an object $X \in \mMor_{l}(\A)$, where $X^0, \dots, X^l \in \E$ and $\omega_{X^l}$ factors through $\tau(\alpha^{l-1} \cdots \alpha^0)$:
\begin{center}
	\begin{tikzcd}[sep=large]
		&&&& X^l \ar[d, "\omega_{X_l}"] \ar[dllll, dashed, "\alpha^l"', bend right=3ex] \\
		\tau X^{0} \ar[r, tail, "\tau \alpha^{0}"] & \tau X^{1} \ar[r, tail, "\tau \alpha^{1}"] & \cdots \ar[r, tail, "\tau \alpha^{l-2}"] & \tau X^{l-1} \ar[r, tail, "\tau \alpha^{l-1}"] & \tau X^l.
	\end{tikzcd}
\end{center}
We denote by $\Fac^\E_{l+1}(\omega)$ the exact category of such factorizations, by ${}^0{\Fac}^\E_{l+1}(\omega)$ its fully exact subcategory of those $X$, where $X^l \in \Proj(\A)$, see \Cref{lem: Fact-fully-exact}.\\
We call $\omega$ \emph{regular} on $\E$ if $\omega_A$ is a monic for all $A \in\E$. In this case, there are \emph{trivial factorizations}
\begin{center}
	\begin{tikzcd}
		\nu^k(A)=(\nu^k(A), \alpha)\colon A \ar[r, equal] & \cdots \ar[r, equal] & A \ar[r, tail, "\omega_A"] & \tau A \ar[r, equal] & \cdots \ar[r, equal] & \tau A
	\end{tikzcd}
\end{center}
in $\Fac^\E_{l+1}(\omega)$, where $\alpha^{k} = \omega_A$, for all $A \in \E$ and $k \in \{0,\dots, l\}$. These relate closely to the projectives and injectives:

\begin{thmA} \label{thmA: Fac-Frobenius}
	Suppose that $\A$ is weakly idempotent complete and $(\tau, \omega)$ regular on $\E$.
	\begin{enumerate}
		\item Suppose that $\E$ has enough projectives. Then $\Fac^\E_{l+1}(\omega)$ has enough projectives. These are the direct summands of direct sums of objects of the form $\nu^k(P) \in \Fac^\E_{l+1}(\omega)$, where $P \in \Proj(\E)$ and $k \in \{0, \dots, l\}$. The same statements hold verbatim for injectives.
		
		\item Suppose that $\E$ is Frobenius. Then $\Fac^\E_{l+1}(\omega)$ is a Frobenius category. If $\E$ has enough $\A$-projectives, then ${}^0{\Fac}^\E_{l+1}(\omega)$ a sub-Frobenius category with the same projective-injectives.
		%Its projective-injectives are the direct summands of direct sums of objects of the form $\nu^k(P) \in {}^0{\Fac}^\E_{l+1}(\omega)$, where $P \in \Proj(\A)$ and $k \in \{0, \dots, l\}$.
		In particular, there is a fully faithful triangle functor ${}^0{\ul{\Fac}}^\E_{l+1}(\omega) \hookrightarrow \ul\Fac^\E_{l+1}(\omega)$.
	\end{enumerate}
\end{thmA}

In the module case, \Cref{thmA: Fac-Frobenius} recovers a result of Sun and Zhang {\cite[Prop.~3.4]{SZ26}}. Their proof relies on a correspondence of Gorenstein projectives under Frobenius functors between Abelian categories. Our proof involves the left and right adjoints of the functors $\nu^k \colon \E \to \Fac^\E_{l+1}(\omega)$, which are given in terms of the projections $\pi^j \colon X \mapsto X^j$ for $j = k$ and $j=k+1$, respectively, see \Cref{lem: nu-adjunction}.

\smallskip

The (generalized) cokernel functor sends an $\E$-factorization $X=(X,\alpha)$ to the chain of monomorphisms
\[
\begin{tikzcd}
	\cok(X) \colon \; U^{1} \ar[r, tail] & U^{2} \ar[r, tail] & \cdots \ar[r, tail] & U^{l-1} \ar[r, tail] & U^l
\end{tikzcd}
 \]
in $\mMor_{l-1}(\A/\omega)$, defined by a commutative diagram
\[ 
	\begin{tikzcd}[sep={17.5mm,between origins}]
		X^{0} \ar[r, tail, "\alpha^{0}"] \ar[d, two heads] \ar[rd, phantom, "\square"] & X^{1} \ar[r, tail, "\alpha^{1}"] \ar[d, two heads] \ar[rd, phantom, "\square"] & X^{2} \ar[r, tail, "\alpha^{2}"] \ar[d, two heads] & \cdots \ar[r, tail, "\alpha^{l-2}"] \ar[rd, phantom, "\square"] & X^{l-1} \ar[r, tail, "\alpha^{l-1}"] \ar[d, two heads] \ar[rd, phantom, "\square"] & X^l \ar[d, two heads] \\
		0 \ar[r, tail] & U^{1} \ar[r, tail] & U^{2} \ar[r, tail] \ar[ru, phantom, "\square"] & \cdots \ar[r, tail] & U^{l-1} \ar[r, tail] & U^l,
	\end{tikzcd}
\]
obtained by successive pushouts. This leads to our main results, formulated under \Cref{asn}:

\begin{thmA} \label{thmA: main-1}
		Suppose that $\E$ has enough $\A$-projectives. Then the cokernel functor induces an equivalence of exact categories
		\begin{center}
			\begin{tikzcd}
				{}^0{\Fac}^\E_{l+1}(\omega)/\nu^l(\Proj(\E)) \ar[r, "\simeq"] & \mMor_{l-1}(\E_\omega).
			\end{tikzcd}
		\end{center}
\end{thmA}

\begin{thmA} \label{thmA: main-2}
	Suppose that $\A$ is weakly idempotent complete, that $\E$ has enough $\A$-projectives, that $(\tau,\omega)$ regular on $\E$, and that both $\E$ and $\E_\omega$ are Frobenius. 
	Then the cokernel functor induces a triangle equivalence
	\begin{center}
		\begin{tikzcd}
			 {}^0{\ul{\Fac}}^\E_{l+1}(\omega) \ar[r, "\simeq"] & \ul \mMor_{l-1}(\E_\omega).
		\end{tikzcd}
	\end{center}
\end{thmA}
In the module case, \Cref{thmA: main-1,thmA: main-2} recover results of Sun and Zhang {\cite[Cor.~4.5, Thm.~4.6]{SZ26}}. \\
To prove these, we express $\Fac^\E_{l+1}(\omega)$ as a suitable diagram category. It is equivalent to the comma category of the cokernel and projection functors
\begin{center}
	\begin{tikzcd}
		\Fac^\E_{2}(\omega) \ar[r, "\cok"] & \E_\omega & \mMor_{l-1}(\E_\omega), \ar[l, "\pi^{l-1}"']
	\end{tikzcd}
\end{center}
see \Cref{prp: comma-equiv}. \Cref{thmA: main-2} is the announced categorical version of Yoshino's result.

\smallskip

\Cref{thmA: Fac-Frobenius,thmA: main-1,thmA: main-2} correspond to \Cref{prp: Fact-enough}, \Cref{cor: Fac-mMor-exact}, and \Cref{thm: coker-equiv}, respectively, in the main part.

\section{Categories of monomorphisms}

In this section, we review preliminaries on monomorphism categories in the context of exact and triangulated categories. All (sub)categories and functors considered are assumed to be (full) additive. By the image of a functor we mean its full image. Admissible monics and epics in exact categories are represented by $\monic{\!}{\!}$ and $\epic{\!}{\!}$, respectively.

\begin{prp}[{\cite[Prop.~2.9]{Buh10}}] \label{prp: Buehler2.9}
	In an exact category, finite direct sums of short exact sequences are again short exact. In particular, any split short exact sequence is short exact. \qed
\end{prp}

\begin{dfn}
	Consider a short exact sequence $\begin{tikzcd}[cramped, sep=small] A \ar[r, tail, "i"] & B \ar[r, two heads, "p"] & C \end{tikzcd}$ in an exact category. Another short exact sequence $\begin{tikzcd}[cramped, sep=small] C \ar[r, tail, "j"] & B \ar[r, two heads, "q"] & A \end{tikzcd}$ is called a \textbf{reverse} of $(i,p)$ if $j$ is right-inverse to $p$, $q$ left-inverse to $i$, and $iq+jp = \id_B$.
\end{dfn}

\begin{lem} \label{lem: reverse}
	In an exact category, a short exact sequence $(i,p)$ admits a reverse $(j,q)$ if and only if $i$ has a left-inverse or $q$ a right-inverse, or, equivalently, if $(i,p)$ is split. In this case, $j$ and $q$ determine each other, and $(i,p)$ is a reverse of $(j,q)$. In addition, reverses are compatible with morphisms of short exact sequences in the following sense: Consider two split short exact sequences $(i, p)$ and $(i',p')$ with reverses $(j, q)$ and $(j',q')$, respectively, and a commutative diagram
	\begin{center}
		\begin{tikzcd}[sep={17.5mm,between origins}]
			A \ar[r, tail, "i"] \ar[d, "a"] 
			& B \ar[r, two heads, "p"] \ar[d, "b"] 
			& C \ar[d, "c"] \\
			A'  \ar[r, tail, "i'"] 
			& B' \ar[r, two heads, "p'"] 
			& C'.
		\end{tikzcd}
	\end{center}	
	Then, in the diagram
	\begin{center} 
		\begin{tikzcd}[sep={17.5mm,between origins}]
			A \ar[d, "a"] & B \ar[l, two heads, "q"'] \ar[d, "b"] & C \ar[l, tail, "j"'] \ar[d, "c"] \\
			A'  & \ar[l, two heads, "q'"'] B'  & \ar[l, tail, "j'"']  C',
		\end{tikzcd}
	\end{center}	
	the left square commutes if and only if the right one does. \qed
\end{lem}

\begin{prp}[{\cite[Prop.~2.12]{Buh10}}] \label{prp: Buehler2.12} \
	\begin{enumerate}[leftmargin=*]
		\item \label{prp: Buehler2.12-push} For a square
		\begin{center}
			\begin{tikzcd}[sep={17.5mm,between origins}]
				A \ar[r, tail, "i"] \ar[d, "f"] & B \ar[d, "f'"] \\
				A' \ar[r, tail, "i'"] & B'
			\end{tikzcd}
		\end{center}
		in an exact category, the following statements are equivalent:
		\begin{enumerate}[label=(\arabic*)]
			\item \label{prp: Buehler2.12-push-1} The square is a pushout.
			\item \label{prp: Buehler2.12-push-2} The square is bicartesian.
			\item \label{prp: Buehler2.12-push-3} The sequence \begin{tikzcd}[sep=large] A \ar[r, "\begin{pmatrix} i \\ -f \end{pmatrix}", tail] & B \oplus A' \ar[r, "\begin{pmatrix} f' \hspace{2mm} i' \end{pmatrix}", two heads] & B' \end{tikzcd} is short exact.
			\item \label{prp: Buehler2.12-push-4} The square is part of a commutative diagram
			\begin{center}
				\begin{tikzcd}[sep={17.5mm,between origins}]
					A \ar[r, tail, "i"] \ar[d, "f"] & B \ar[d, "f'"] \ar[r, two heads] & C \ar[d, equal] \\
					A' \ar[r, tail, "i'"] & B' \ar[r, two heads] & C.
				\end{tikzcd}
			\end{center}
		\end{enumerate}
		
		\smallskip
		
		\item \label{prp: Buehler2.12-pull} For a square
		\begin{center}
			\begin{tikzcd}[sep={17.5mm,between origins}]
				A \ar[r, two heads, "p'"] \ar[d, "g'"] & B \ar[d, "g"] \\
				A' \ar[r, two heads, "p"] & B'
			\end{tikzcd}
		\end{center}
		in an exact category, the following statements are equivalent:
		\begin{enumerate}[label=(\arabic*)]
			\item The square is a pullback.
			\item The square is bicartesian.
			\item The sequence \begin{tikzcd}[sep=large] A \ar[r, "\begin{pmatrix} p' \\ g' \end{pmatrix}", tail] & B \oplus A' \ar[r, "\begin{pmatrix} - g \hspace{2mm} p \end{pmatrix}", two heads] & B' \end{tikzcd} is short exact.
			\item The square is part of a commutative diagram
			\begin{center}
				\begin{tikzcd}[sep={17.5mm,between origins}]
					K \ar[r, tail] \ar[d, equal] & A \ar[r, two heads, "p'"] \ar[d, "g'"] & B \ar[d, "g"] \\
					K \ar[r, tail] & A' \ar[r, two heads, "p"] & B'.
				\end{tikzcd}
			\end{center}
		\end{enumerate}
	\end{enumerate}
	\qed 
\end{prp}

\begin{lem}[{\cite[Prop.~3.1]{Buh10}}] \label{lem: Buh3.1}
	In an exact category, any morphism $(a,b,c) \colon (i,p) \to (i',p')$ between short exact sequences $\begin{tikzcd}[cramped, sep=small] A \ar[r, tail, "i"] & B \ar[r, two heads, "p"] & C \end{tikzcd}$ and $\begin{tikzcd}[cramped, sep=small] A' \ar[r, tail, "i'"] & B' \ar[r, two heads, "p'"] & C' \end{tikzcd}$ factors through a short exact sequence $\begin{tikzcd}[cramped, sep=small] A' \ar[r, tail] & \tilde B \ar[r, two heads] & C \end{tikzcd}$ as follows:
	\begin{center}
		\begin{tikzcd}[sep={17.5mm,between origins}]
			A \ar[r, tail, "i"] \ar[d, "a"] \ar[rd, phantom, "\square"]
			& B \ar[r, two heads, "p"] \ar[d, dashed] 
			& C \ar[d, equal] \\
			A' \ar[d, equal] \ar[r, tail, dashed] & \tilde B \ar[r, two heads, dashed] \ar[d, dashed] \ar[rd, phantom, "\square"] & C \ar[d, "c"] \\
			A'  \ar[r, tail, "i'"] 
			& B' \ar[r, two heads, "p'"] 
			& C'
		\end{tikzcd}
	\end{center}
	\qed
\end{lem}

\begin{lem} \label{lem: restrict-splitting}
	In an exact category $\A$, consider a commutative diagram	
	\begin{equation} \label{diag: restrict-splitting}
		\begin{tikzcd}[sep={17.5mm,between origins}]
			A \ar[r, tail, "i"] \ar[d, hook, "a"] 
			& B \ar[r, two heads, "p"] \ar[d, hook, "b"] 
			& C \ar[d, hook, "c"] \\
			A'  \ar[r, tail, "i'"] 
			& B' \ar[r, two heads, "p'"] 
			& C'.
		\end{tikzcd}
	\end{equation}
	Suppose that $a,b,c$ are monics and the lower row splits. Consider any reverse $(j',q')$ of $(i', p')$ and morphisms $j \colon C \to B$ and $q \colon B \to A$ which fit into the commutative diagram
	\begin{equation*}
		\begin{tikzcd}[sep={17.5mm,between origins}]
			A \ar[d, hook, "a"] 
			& B \ar[l, "q"'] \ar[d, hook, "b"] 
			& C \ar[d, hook, "c"] \ar[l, "j"'] \\
			A'  
			& B' \ar[l, two heads, "q'"'] 
			& C'. \ar[l, tail, "j'"']
		\end{tikzcd}
	\end{equation*}
	Then $(j,q)$ is a reverse of $(i,p)$.
\end{lem}

\begin{proof}
	We compute
	\[b \circ (iq+jp) = biq+bjp = i'aq+j'cp=i'q'b+j'p'b = (i'q'+j'p') \circ b = b,\]
	which implies $iq+jp=\id_B$ since $b$ is a monic. Similarly,
	\[aqi=q'bi=q'i'a=a \hspace{5mm} \textup{and} \hspace{5mm} cpj=p'bj=p'j'c=c\]
	imply $qi=\id_A$ and $pj=\id_C$.
\end{proof}

\begin{lem} \label{lem: lift-splitting}
	In an exact category $\A$, consider a commutative diagram
	\begin{equation} \label{diag: lift-splitting}
		\begin{tikzcd}[sep={17.5mm,between origins}]
			A \ar[r, tail, "i"] \ar[d, two heads, "a"] 
			& B \ar[r, two heads, "p"] \ar[d, "b"] 
			& C \ar[d, "c"] \\
			A'  \ar[r, tail, "i'"] 
			& B' \ar[r, two heads, "p'"] 
			& C'.
		\end{tikzcd}
	\end{equation}	
	Suppose that $a$ is an admissible epic, that $C \in \Proj(\A)$ and that the lower row splits. Then for any reverse $(j',q')$ of $(i', p')$, there is a reverse $(j,q)$ of $(i,p)$ which fits into the following commutative diagram:
	\begin{equation*} %\label{diag: reverse-splitting}
		\begin{tikzcd}[sep={17.5mm,between origins}]
			A \ar[d, two heads, "a"] & B \ar[l, two heads, dashed, "q"'] \ar[d, "b"] & C \ar[l, tail, dashed, "j"'] \ar[d, "c"] \\
			A'  & \ar[l, two heads, "q'"'] B'  & \ar[l, tail, "j'"']  C'
		\end{tikzcd}
	\end{equation*}	
\end{lem}

\begin{proof}
	By \Cref{lem: Buh3.1}, there is a commutative diagram
	\begin{center}
		\begin{tikzcd}[sep={17.5mm,between origins}]
			A \ar[r, tail, "i"] \ar[d, two heads, "a"] 
			& B \ar[r, two heads, "p"] \ar[d, two heads, dashed, "\tilde a"] 
			& C \ar[d, equal] \\
			A' \ar[d, equal] \ar[r, tail, dashed, "\tilde i"] & \tilde B \ar[r, two heads, dashed, "\tilde p"] \ar[d, dashed, "\tilde c"] & C \ar[d, "c"] \\
			A'  \ar[r, tail, "i'"] 
			& B' \ar[r, two heads, "p'"] 
			& C',
		\end{tikzcd}
	\end{center}
	in $\A$, where $(\tilde i, \tilde p)$ is a short exact sequence and $\tilde c \tilde a = b$. The left-inverse $\tilde q := q' \tilde c$ of $\tilde i$ defines a reverse $(\tilde j, \tilde q)$ of $(\tilde i, \tilde p)$. Due to $C \in \Proj(\A)$, there is a morphism $j \colon C \to B$ such that $\tilde j = \tilde a j$. This right-inverse of $p$ defines a reverse $(j,q)$ of $(i,p)$. In the diagram
	\begin{center}
		\begin{tikzcd}[sep={17.5mm,between origins}]
			A \ar[d, two heads, "a"] 
			& B \ar[l, two heads, "q"'] \ar[d, two heads, "\tilde a"] 
			& C \ar[d, equal] \ar[l, tail, "j"'] \\
			A' \ar[d, equal] & \tilde B \ar[l, two heads, "\tilde q"'] \ar[d, "\tilde c"] & C \ar[d, "c"] \ar[l, tail, "\tilde j"'] \\
			A'  
			& B' \ar[l, two heads, "q'"'] 
			& C', \ar[l, tail, "j'"']
		\end{tikzcd}
	\end{center}
	the upper right and the lower left square commute by construction. Then so do the other two, due to \Cref{lem: reverse}.
\end{proof}

\begin{lem}[Noether lemma, {\cite[Ex.~3.7]{Buh10}}] \label{lem: Noether}
	Any solid commutative diagram
	
	\begin{center}
		\begin{tikzcd}[sep={17.5mm,between origins}]
			A' \ar[r, tail] \ar[d, tail] & B' \ar[r, two heads] \ar[d, tail] & C' \ar[d, tail, dashed]\\
			A \ar[r, tail] \ar[d, two heads] & B \ar[r, two heads] \ar[d, two heads] & C \ar[d, two heads, dashed]\\
			A'' \ar[r, tail] & B'' \ar[r, two heads] & C'' 
		\end{tikzcd}
	\end{center}
	in an exact category with short exact rows and columns can be uniquely completed by a dashed short exact sequence. \qed
\end{lem}

\begin{dfn}
	A functor $\E' \to \E$ of exact categories is \textbf{(fully) exact} if it preserves (and reflects) short exact sequences. A subcategory $\E'$ of an exact category $\E$ is called \textbf{(fully) exact} if it is an exact category itself, and if the inclusion functor is (fully) exact.\footnote{Bühler uses the term \emph{fully exact} for the stronger notion of extension-closedness, see {\cite[Lem.~10.20]{Buh10}}.} 
\end{dfn}

\begin{lem}[{\cite[Lem.~1.19]{FS24}}] \label{lem: fullyexact} Let $\E'$ be a subcategory of an exact category $\E$. Suppose that one of the following conditions holds:
	\begin{enumerate}
		\item \label{lem: fullyexact-extclosed} $\E'$ is an extension-closed subcategory of $\E$.
		\item \label{lem: fullyexact-kerclosed} $\E'$ is closed in $\E$ under kernels of admissible epics and cokernels of admissible monics.
	\end{enumerate}
	Then $\E'$ is fully exact in $\E$, with the exact structure given by short exact sequences in $\E$ with objects in $\E'$. \qed
\end{lem}

\begin{prp}[{\cite[Ex.~13.5, Prop.~11.3, Cor.~11.4]{Buh10}}]
	The subcategory $\Proj(\E)$ of projective objects of an exact category $\E$ is closed under kernels of admissible epics and the subcategory $\Inj(\E)$ of injective objects under cokernels of admissible monics. In particular, both are fully exact in $\E$ and carry the split exact structure. \qed
\end{prp}

We make use of the following terminology, due to Demonet and Iyama {\cite{DI16}}:

\begin{dfn}
	An exact functor $F \colon \E' \to \E$ is \textbf{extension-injective}, \textbf{-surjective}, or \textbf{-bijective} if, for all $X, Z \in \E'$, the induced group homomorphism
	\begin{equation*}
		\Ext^1_{\E'}(X, Z) \longrightarrow \Ext^1_{\E}(F(X), F(Z))
	\end{equation*}
	is injective, surjective, or bijective, respectively. 
\end{dfn}

For an equivalence, the terms \emph{fully exact} and \emph{extension-bijective} coincide, see also {\cite[Def.~3.2]{DI16}}:

\begin{lem} \label{lem: ext-bij}
	Let $F \colon \E' \to \E$ be an exact functor which is an equivalence with quasi-inverse $G$. Then the following are equivalent:
	\begin{enumerate}[label=(\arabic*)]
		\item $F$ is fully exact.
		\item $F$ is extension-bijective.
		\item $G$ is exact.
	\end{enumerate}
	In this case, $F$ is an \textbf{equivalence of exact categories}. \qed
\end{lem}

\begin{dfn}
	For an exact category $\E'$, consider a full, essentially surjective functor $F \colon \E' \to \E$. The \textbf{image (candidate) exact structure} induced by $F$ on $\E$ consists of all composable pairs of morphisms, isomorphic to $(F(i'),F(p'))$ for some short exact sequence $(i',p')$ of $\E'$. In general, this is not an exact structure.
	%Any such $F$ induces an equivalence $\overline F \colon \E'/\cI' \to \E$, where $\cI'$ is the ideal of $\E'$ consisting of morphisms sent to zero under $F$:
	%\begin{center}
	%	\begin{tikzcd}[row sep={15mm,between origins}, column sep={10mm,between origins}]
	%		& \E' \ar[ld, "\pi"'] \ar[rd, "F"] & \\
	%		\E'/\cI' \ar[rr, "\overline F", "\simeq"'] && \E
	%	\end{tikzcd}
	%\end{center}
	%Thus, up to an equivalence of its domain, $F$ is the canonical quotient functor $\pi \colon \E' \to \E'/\cI'$. The image (exact) structure on $\E'/\cI'$ is called \textbf{quotient (exact) structure.}
\end{dfn}

\begin{lem} \label{lem: image-structure}
	Let $F \colon \E' \to \E$ be a full, essentially surjective functor between exact categories. Then $F$ is extension-surjective if and only if $\E$ carries the image exact structure induced by $F$. \qed
\end{lem}

\begin{prp} \label{prp: image-structure}
	Let $G \colon \E' \xrightarrow{\simeq} \E$ be an equivalence. Suppose that $\E'$ is an exact category and equip $\E$ with the image candidate exact structure induced by $G$.
	\begin{enumerate}
		\item Then $\E$ is an exact category and $G$ an equivalence of exact categories.
		
		\item For another exact category $\E''$, consider two full, essentially surjective functors $F' \colon \E'' \to \E'$ and $F \colon \E'' \to \E$ which fit into a diagram
		\begin{equation*}
			\begin{tikzcd}[row sep={15mm,between origins}, column sep={10mm,between origins}]
				& \E'' \ar[ld, "F'"'] \ar[rd, "F"] & \\
				\E' \ar[rr, "G", "\simeq"'] && \E,
			\end{tikzcd}
		\end{equation*}
		commutative up to natural isomorphisms. Suppose that $F'$ is extension-surjective. Then so is $F$. In particular, $\E$ carries the image exact structure induced by $F$, and $F$ is extension-injective if and only if $F'$ is so. \qed
	\end{enumerate}
\end{prp}

\begin{dfn} \label{dfn: quot}
	An \textbf{ideal} $\cI$ of a category $\A$ is a class of morphisms, closed under under pre- and postcomposition with arbitrary morphisms, such that $\cI(A,B) := \cI \cap \Hom_\A(A, B)$ is a (normal) subgroup of $\Hom_\A(A,B)$ for each $A, B \in \A$. The \textbf{quotient category} $\A/\cI$ has the same objects as $\A$ and groups of morphisms \[\Hom_{\A/\cI}(A,B) := \Hom_\A(A, B)/\cI(A, B)\]
	for all $A, B \in \A$. Composition is defined on representatives.
	If $\cI$ is the class of morphisms which factor through objects of a subcategory $\B$ of $\A$, which is closed under biproducts, the quotient category is denoted by $\A/\B$.\\
	For an exact category $\E$, the quotient categories $\ul \E := \E/\Proj(\E)$ and $\ol \E := \E/\Inj(\E)$ are referred to as the \textbf{projectively} and \textbf{injectively stable category} of $\E$, respectively.
\end{dfn}

\begin{ntn}
	Given a subcategory $\cS$ of a category $\A$, let $\langle \cS \rangle$ denote the smallest subcategory of $\A$ containing $\cS$ which is closed under biproducts.
\end{ntn}

\begin{rmk} \label{rmk: factor-through-biproduct}
Consider full subcategories $\cS$ and $\T$ of an additive category $\A$, closed under biproducts. A morphism $f\colon A \to B$ in $\A$ is zero in $\A/\langle \cS \cup \T \rangle$ if and only if there are objects $S \in \cS$, $T \in \T$ and morphisms $r \colon S \to B$, $t \colon T \to B$, $s \colon A \to S$, and $U \colon A \to T$ in $\A$ such that $f=rs+tu$.
\end{rmk}

\begin{rmk}
	Any full, essentially surjective functor $F \colon \E' \to \E$ induces an equivalence $\overline F \colon \E'/\cI' \to \E$, where $\cI'$ is the ideal of $\E'$ consisting of morphisms sent to zero under $F$. The diagram
	\begin{equation*}
		\begin{tikzcd}[row sep={15mm,between origins}, column sep={10mm,between origins}]
				& \E' \ar[ld, "\pi"'] \ar[rd, "F"] & \\
				\E'/\cI' \ar[rr, "\overline F", "\simeq"'] && \E,
			\end{tikzcd}
	\end{equation*}
	where $\pi \colon \E' \to \E'/\cI'$ is the canonical quotient functor, commutes. Thus, up to an equivalence of its codomain, any such $F$ is a canonical quotient functor.
	%The \textbf{quotient (exact) structure} on $\E'/\cI'$ is the image (exact) structure, induced by $\pi$.
\end{rmk}

%\begin{rmk} \label{rmk: induced-ext-bij}
%	Let $F \colon \E' \to \E$ be an exact functor and $\cI'$ an ideal of $\E'$ with $F(\cI')=0$. Suppose that $\E'/\cI'$ is an exact category such that the canonical quotient functor $\E' \to \E'/\cI'$ is extension-surjective. Then the unique induced functor $\overline F \colon \E'/\cI' \to \E$ is exact. It is extension-bijective if and only if $F$ is so.
%\end{rmk}

%There is the following partial converse to \Cref{rmk: induced-ext-bij}:

We specialize \Cref{prp: image-structure} to the case of quotients:

\begin{prp} \label{prp: homomorphism-thm-reversed}
	Let $F \colon \E' \to \E$ be a full, essentially surjective,  extension-surjective functor between exact categories. Consider the ideal $\cI'$ of $\E'$ consisting of morphisms sent to zero under $F$, that is, such that the induced functor $\overline F \colon \E'/\cI' \to \E$ is an equivalence. Then the canonical quotient functor $\pi \colon \E' \to \E'/\cI'$ is extension-surjective and $\overline F$ an equivalence of exact categories. In particular, $\E'/\cI'$ carries the image exact structure induced by $\pi$, and $\pi$ is extension-bijective if and only if $F$ is so.
\end{prp}

\begin{proof}
	Pick a quasi-inverse $G$ of $\overline F$. It fits into a diagram
	\begin{equation*}
		\begin{tikzcd}[row sep={15mm,between origins}, column sep={10mm,between origins}]
			& \E' \ar[rd, "\pi"] \ar[ld, "F"'] & \\
			\E \ar[rr, "G", "\simeq"'] && \E'/\cI',
		\end{tikzcd}
	\end{equation*}
	commutative up to natural isomorphisms. The claim follows from \Cref{prp: image-structure}.
\end{proof}

\begin{dfn}
	We say that an exact subcategory $\E'$ of $\E$ has \textbf{enough $\E$-projectives} if there is an admissible epic $P \twoheadrightarrow X$ in $\E'$ with $P \in \Proj(\E)$ for each $X \in \E'$. Having \textbf{enough $\E$-injectives} is defined dually. If this holds for $\E' = \E$, one says that $\E$ has enough \textbf{enough projectives} or \textbf{enough injectives}, respectively.
\end{dfn}

\begin{dfn} A \textbf{Frobenius (exact) category} is an exact category $\F$ with enough projectives and injectives such that $\Proj(\F) = \Inj(\F)$. In this case, one speaks of \textbf{projective-injective} objects and calls $\ul \F = \ol \F$ the \textbf{stable category} of $\F$. By a \textbf{sub-Frobenius category} of a Frobenius category $\F$, we mean an exact subcategory $\F'$ which has enough $\F$-projectives and enough $\F$-injectives. This terminology is justified by \Cref{lem: sub-Frobenius}.\ref{lem: sub-Frobenius-b}.
\end{dfn}

\begin{thm}[{\cite[Thm.~2.6]{Hap88}}] \label{thm: stable-Frobenius}
	The stable category of a Frobenius category is triangulated. \qed
\end{thm}

\begin{prp}[{\cite[Prop.~7.3]{IKM16}}] \label{prp: stable-functor-triang}
	Any exact functor $F\colon \F' \to \F$ of Frobenius categories preserving projective-injectives induces a triangle functor $\ul F\colon \ul \F' \to \ul \F$ of the respective stable categories.\qed
\end{prp}

\begin{lem}[{\cite[Lem.~1.37]{FS24}}] \label{lem: sub-Frobenius}
	Let $\E'$ be an exact subcategory of $\E$.
	\begin{enumerate}
		\item \label{lem: sub-Frobenius-a} If $\E'$ has enough $\E$-projectives, then $\Proj(\E') = \Proj(\E) \cap \E'$, and the canonical functor $ \underline \E' \to \underline \E$ is fully faithful.
		
		\item \label{lem: sub-Frobenius-b} If $\F'$ is a sub-Frobenius category of $\F$, then $\F'$ is Frobenius, and the canonical functor $ \underline \F' \to \underline \F$ is a fully faithful triangle functor. \qed
	\end{enumerate}
\end{lem}

\begin{dfn}[{\cite[Def.~7.2]{Buh10}}]
	An additive category is called \textbf{weakly idempotent complete} if any (co)retraction has a (co)kernel.
\end{dfn}

\begin{prp}[{\cite[Cor.~7.5, Prop.~7.6]{Buh10}}] \label{prp: WIC}
	For an exact category $\E$, the following are equivalent:
	\begin{enumerate}
		\item $\E$ is weakly idempotent complete.
		\item Any retraction is an admissible epic.
		\item Any coretraction is an admissible monic.
		\item If a composition $gf$ of morphisms in $\E$ is an admissible epic, then so is $g$.
		\item \label{prp: WIC-comp-monic} If a composition $gf$ of morphisms in $\E$ is an admissible monic, then so is $f$.\qed
	\end{enumerate}
\end{prp}

\begin{ntn}[{\cite[Def.~3.1]{BM24}}, {\cite[Def~4.1]{IKM17}}] \label{ntn: mMor} 
	Let $\E$ be an exact category and $l \in \NN$.
	\begin{enumerate}
		\item \label{ntn: mMor-category} Let $\Mor_{l}(\E)$ denote the category of diagrams
		\[
		\begin{tikzcd}
			X=(X, \alpha)\colon \; X^0 \ar[r, "\alpha^0"] & X^1 \ar[r, "\alpha^1"] & \cdots \ar[r, "\alpha^{l-1}"] & X^l
		\end{tikzcd}
		\]
		of type $A_{l+1}$ in $\E$ with the termwise exact structure, where $A_{l+1}$ is the unidirectional linear quiver with $l+1$ vertices.
		
		\item \label{ntn: mMor-pi} For $k \in \{0, \dots, l\}$, let $\pi^k \colon \Mor_l(\E) \to \E$ denote the exact functor which sends $X \in \Mor_l(\E)$ to $X^k \in \E$. Note that $\pi^l$ restricts to a faithful functor $\mMor_{l}(\E) \to \E$.
		
		\item \label{ntn: mMor-iota} Let $\iota = \iota_l \colon \Mor_{l}(\E) \to \Mor_{l+1}(\E)$ denote the fully faithful, fully exact functor which sends $X \in \Mor_l(\E)$ to
		\[
		\begin{tikzcd}
			0 \ar[r] & X^0 \ar[r, "\alpha^0"] & X^1 \ar[r, "\alpha^1"] & \cdots \ar[r, "\alpha^{l-1}"] & X^l.
		\end{tikzcd}
		\]
		
		\item By $\mMor_l(\E)$ and $\smMor_l(\E)$ we denote the subcategories of $\Mor_l(\E)$, where all arrows are admissible or split monics, respectively. For a subcategory $\E'$ of $\E$, we set \[\Mor^{\textup{m}_\E}_l(\E') := \mMor_{l}(\E) \cap \Mor_{l}(\E').\]
		
		\item \label{ntn: mMor-mu} Given an object $A \in \E$ and $n \in \{1, \dots, l+1\}$, we define the \textbf{trivial chain of monics}
		\begin{center}
			\begin{tikzcd}
				\mu_n(A)\colon \; 0 \ar[r, equal] & \cdots \ar[r, equal] & 0 \ar[r, tail] & A^{l-n+1} \ar[r, equal] & \cdots \ar[r, equal] & A^{l}
			\end{tikzcd}
		\end{center}
		in $\smMor_{l}(\E)$ by $A^k := A$ for $k \in \{l-n+1, \dots, l\}$.
	\end{enumerate}
\end{ntn}

\begin{thm}[{\cite[Props.~3.5, 3.9, 3.11, Thm.~3.12]{BM24}}] \label{thm: mMor}
	Let $\E$ be an exact category.
	\begin{enumerate}
		\item \label{thm: mMor-exact} The category $\mMor_l(\E)$ is a fully exact subcategory of $\Mor_l(\E)$.
		
		\item \label{thm: mMor-Proj} We have $\Proj(\mMor_{l}(\E))=\smMor_l(\Proj(\E))$ and $\Inj(\mMor_{l}(\E))=\smMor_l(\Inj(\E))=\mMor_l(\Inj(\E))$.
		
		\item \label{thm: mMor-enough} If $\E$ has enough projectives, or, injectives, then so has $\mMor_l(\E)$.
	\end{enumerate}
	In particular, if $\F$ is a Frobenius category, then so is $\mMor_l(\F)$, and $\Proj(\mMor_{l}(\F))=\mMor_l(\Proj(\F))$. \qed
\end{thm}

\begin{ntn}
	For a Frobenius category $\F$ and $l \in \NN$, we denote the stable category of $\mMor_{l}(\F)$ by $\ul \mMor_{l}(\F)$. It is a triangulated category, see \Cref{thm: stable-Frobenius}.
\end{ntn}

\begin{con}[Projectives in $\mMor_{l}(\E)$] \label{con: mMor-enough}
	Let $\E$ be an exact category with enough projectives, $l \in \NN$, and $X=(X, \alpha) \in \mMor_{l}(\E)$. For each $k \in \{0, \dots, l\}$, choose an admissible epic $q^k \colon \epic{Q^k}{X^k}$ in $\E$ with $Q^k \in \Proj(\E)$. Then
	\[P := \bigoplus_{k=0}^l \mu_{l-k+1}(Q^k)\in \Proj(\mMor_l(\F)),\]
	and there is an admissible epic $p=(p^k)_{k=0,\dots,l}\colon \epic P X$ in $\mMor_l(\F)$, defined by
		\[p^k := \begin{pmatrix}
			\alpha^{k-1} \cdots \alpha^0 q^0 & \cdots & \alpha^{k-1} q^{k-1} & q^k
		\end{pmatrix} \colon P^k = Q^0 \oplus \dots \oplus Q^k \to X^k.\]
\end{con}

\section{Hypersurface categories}

In this section, we introduce twisted homotheties to generalize regular normal elements of a ring in a categorical setting. The hypersurface category associated to a twisted homothety then extends the concept of modules over a hypersurface ring to exact categories.

\begin{dfn}
	By a \textbf{twisted homothety} $(\tau, \omega)$ on an additive category $\A$ we mean an additive automorphism $\tau$ of $\A$ together with a natural transformation $\omega \colon \id_\A \to \tau$ such that $\omega \tau = \tau \omega$. If $\A$ is an exact category, we require $\tau$ to be fully exact.
\end{dfn}

\begin{dfn}
	Let $(\tau, \omega)$ be a twisted homothety on an additive category $\A$.
	We define the \textbf{hypersurface category} $\A / \omega$ as the subcategory of $\A$ consisting of all objects $X \in \A$ with $\omega_X = 0$. It is clearly replete and preadditive, and it contains the zero object of $\A$. The category $\A / \omega$ is then additive, since $\omega_{X \oplus Y} = \omega_X \oplus \omega_Y$ for all $X, Y \in \A$.
\end{dfn}

\begin{rmk} \label{rmk: tau-A/omega}
	Any twisted homothety $(\tau, \omega)$ on an additive category $\A$ restricts to $\A/\omega$. Indeed, $\tau(\A/\omega) = \A/\omega$, since $\omega_{\tau A} = \tau(\omega_A) = 0$ if and only if $\omega_A=0$ for each $A \in \A$.
\end{rmk}

\begin{lem} \label{lem: A/omega-fully-exact}
	Let $(\tau, \omega)$ be a twisted homothety on an exact category $\A$. If \begin{tikzcd}[cramped, sep=small] X \ar[r, tail, "i"] & Y \ar[r, two heads, "p"] & Z \end{tikzcd} is a short exact sequence in $\A$ with $Y \in \A/\omega$, then $X,Z \in \A/\omega$. In particular, $\A/\omega$ is a fully exact subcategory of $\A$.
\end{lem}

\begin{proof}
	Consider the diagram
	\begin{center}
		\begin{tikzcd}[sep=large]
			X \ar[r, tail, "i"] \ar[d, "\omega_X"] & Y \ar[r, two heads, "p"] \ar[d, "\omega_Y \, = \, 0"] & Z \ar[d, "\omega_Z"] \\
			\tau X \ar[r, tail, "\tau i"] & \tau Y \ar[r, two heads, "\tau p"] & \tau Z
		\end{tikzcd}
	\end{center}
	with short exact rows. The definitions of monics and epics yield that $\omega_X =0$ and $\omega_Z=0$, respectively. The particular claim follows from \Cref{lem: fullyexact}.\ref{lem: fullyexact-kerclosed}.
\end{proof}

\begin{rmk}
	In general, the subcategory $\A/\omega$ of $\A$ is not extension-closed.
\end{rmk}

\begin{lem} \label{lem: A/omega-coker}
	Let $(\tau, \omega)$ be a twisted homothety on an exact category $\A$ and \begin{tikzcd}[cramped, sep =small] X \ar[r, tail, "i"] & Y \ar[r, two heads, "p"] & Z \end{tikzcd} a short exact sequence  in $\A$. Then $Z \in \A/\omega$ if and only if $\omega_Y$ factors through $\tau i$, or, equivalently, if $\omega_{\tau^{-1}Y}$ factors through $i$.
\end{lem}

\begin{proof}
	Consider the commutative diagram
	
	\begin{center}
		\begin{tikzcd}[sep=large]
			X \ar[r, tail, "i"] \ar[d, "\omega_X"] & Y \ar[r, two heads, "p"] \ar[d, "\omega_Y"] & Z \ar[d, "\omega_Z"] \\
			\tau X \ar[r, tail, "\tau i"] & \tau Y \ar[r, two heads, "\tau p"] & \tau Z
		\end{tikzcd}
	\end{center}
	with short exact rows. Since $p$ is an epic, $\omega_Z=0$ is equivalent to $\tau(p) \omega_Y=\omega_Zp=0$. As $\tau i$ is the kernel of $\tau p$, this holds if and only if $\omega_Y$ factors through $ \tau i$.
\end{proof}

\begin{dfn}
	We call a twisted homothety $(\tau, \omega)$ on an exact category $\A$ \textbf{regular} on a subcategory $\B \subseteq \A$ if $\omega_B$ is an admissible monic in $\A$ for all $B \in \B$.
\end{dfn}

\begin{ntn} \label{ntn: overline-omega}
	Let $(\tau, \omega)$ be a twisted homothety on an exact category $\A$, regular on a subcategory $\B \subseteq \A$. Given any object $B \in \B$, we denote the cokernel of $\omega_{\tau^{-1}B}\colon \tau^{-1}B \rightarrowtail B$ by $\ol \omega_{B}\colon B \twoheadrightarrow \ol {B}$. Since $\omega_{B} = \tau(\omega_{\tau^{-1} B}) \circ \id_B$, we have $\ol B \in \A/\omega$ due to \Cref{lem: A/omega-coker}.
\end{ntn}

\begin{lem} \label{lem: Proj-A/omega}
	Let $(\tau, \omega)$ be a twisted homothety on an exact category $\A$, regular on a subcategory  $\B$ of $\Proj(\A)$. Then $\ol P \in \Proj(\A/\omega)$ for any $P \in \B$.
\end{lem}

\begin{proof}
	Consider an arbitrary admissible epic $p\colon Y \twoheadrightarrow Z$ in $\A/\omega$ and an arbitrary morphism $a\colon \ol P \to X$. Lifting $a \ol \omega_{P}$ along $p$, yields a morphism of the form $b$ as shown in the commutative diagram
	\begin{center}
		\begin{tikzcd}[sep=large]
			\tau^{-1} P \ar[r, "\omega_{\tau^{-1}P}", tail] \ar[d, "\tau^{-1} b", dashed] & P \ar[r, "\ol \omega_{P}", two heads] \ar[d, "b", dashed] & \ol P \ar[d, "a"] \ar[ld, dashed, "c"']\\
			\tau^{-1} Y \ar[r, "\omega_{\tau^{-1}Y}"'] & Y \ar[r, "p"', two heads] & Z.
		\end{tikzcd}
	\end{center}
	Since $\tau^{-1}Y \in \A/\omega$, see \Cref{rmk: tau-A/omega}, $b \omega_{\tau^{-1} P} = \omega_{\tau^{-1} Y} \tau^{-1}(b) =0$ and hence $b$ factors through $\ol \omega_{P}$. The resulting morphism $c$ satisfies $p c \ol \omega_{P} = p b = a \ol \omega_{P}$, and hence $p c = a$ since $\ol \omega_{P}$ is an epic.
\end{proof}

\section{Categories of factorizations}	

In this section, we define factorizations of a twisted homothety $(\tau, \omega)$ over an exact subcategory $\E$ of $\A$ into chains of monics in $\A$. It is shown that such factorizations form an exact category with the termwise exact structure. Under suitable hypotheses, we describe its projectives and injectives and show that it inherits the Frobenius property from $\E$. 

\begin{dfn} \label{dfn: Fac}
	Let $(\tau, \omega)$ be a twisted homothety on an exact category $\A$, $\E = \tau \E$ a fully exact subcategory of $\A$, and $l \in \NN$. By an \textbf{$\boldsymbol \E$-factorization} of $(\tau, \omega)$ with $l+1$ factors we mean an object
	\begin{center}
		\begin{tikzcd}[sep=large]
			X=(X, \alpha)\colon X^{0} \ar[r, tail, "\alpha^{0}"] & X^{1} \ar[r, tail, "\alpha^{1}"] & \cdots \ar[r, tail, "\alpha^{l-2}"] & X^{l-1} \ar[r, tail, "\alpha^{l-1}"] & X^l
		\end{tikzcd}
	\end{center}
	of $\Mor^{\textup{m}_\A}_l(\E)$, where $\omega_{X^l} = \tau(\alpha^{l-1} \cdots \alpha^{0}) \alpha^l$ for some morphism $\alpha^l\colon X^l \to \tau X^{0}$:
	
	\begin{center}
		\begin{tikzcd}[sep=large]
			&&&& X^l \ar[d, "\omega_{X_l}"] \ar[dllll, dashed, "\alpha^l"', bend right=3ex] \\
			\tau X^{0} \ar[r, tail, "\tau \alpha^{0}"] & \tau X^{1} \ar[r, tail, "\tau \alpha^{1}"] & \cdots \ar[r, tail, "\tau \alpha^{l-2}"] & \tau X^{l-1} \ar[r, tail, "\tau \alpha^{l-1}"] & \tau X^l
		\end{tikzcd}
	\end{center}

	Since $\tau(\alpha^{l-1} \cdots \alpha^0)$ is monic, $\alpha^l$ is unique if it exists. Note that $\alpha^k$ is only an admissible monic in $\A$, and its cokernel might not lie in $\E$. We denote the subcategory of $\Mor_{l}(\E)$ consisting of such factorizations by $\Fac^\E_{l+1}(\omega)$. The objects $X \in \Fac^\E_{l+1}(\omega)$ with $X^l \in \Proj(\A)$ form a subcategory ${}^0{\Fac}^\E_{l+1}(\omega)$.
\end{dfn}

\begin{rmk} \label{rmk: Fac-cok-A/omega}
	An object $X=(X, \alpha) \in \Mor^{\textup{m}_\A}_l(\E)$ lies in $\Fac^\E_{l+1}(\omega)$ if an only if the cokernel of $\alpha^{l-1} \cdots \alpha^0$ lies in $\A/\omega$, see \Cref{lem: A/omega-coker}.
\end{rmk}

\begin{rmk} \label{rmk: tau-restricted}
	Let $\tau$ be a fully exact equivalence on an exact category $\A$ and $\E$ a fully exact subcategory of $\A$. If $\E = \tau \E$, then $\tau$ restricts to a fully exact equivalence of $\E$. In this case, $A \in \Proj(\E)$ if and only if $\tau A \in \Proj(\E)$ for any $A \in \E$, and verbatim for injectives.
\end{rmk}

\begin{dfn}
	Let $\A$ be an exact category and $l \in \NN$. The \textbf{contraction} functor
	\[\gamma \colon \mMor_l(\A) \longrightarrow \mMor_1(\A), \; (X, \alpha) \longmapsto \left(\begin{tikzcd}[cramped, sep=large] X^0 \ar[r, tail, "\alpha^{l-1} \cdots \alpha^0"] & X^l\end{tikzcd}\right),\]
	is faithful, essentially surjective, and exact.
\end{dfn}

\begin{rmk}
	Given a twisted homothety $(\tau, \omega)$ on an exact category $\A$ and a fully exact subcategory $\E = \tau \E$ of $\A$, we have $X \in \Fac^\E_{l+1}(\omega)$ if and only if $\gamma X \in \Fac^\E_2(\omega)$ for any $X \in \Mor^{\textup{m}_\A}_l(\E)$, see \Cref{rmk: Fac-cok-A/omega}.
\end{rmk}

\begin{lem} \label{lem: Fact-fully-exact}
	Let $(\tau, \omega)$ be a twisted homothety on an exact category $\A$, $\E = \tau \E$ a fully exact subcategory, and $l \in \NN$.
	\begin{enumerate}
		\item The subcategory $\Mor^{\textup{m}_\A}_l(\E)$ is closed in $\Mor_l(\E)$ under extensions.
		\item \label{lem: Fact-fully-exact-co-kernels} The subcategory $\Fac^\E_{l+1}(\omega)$ is closed in $\Mor^{\textup{m}_\A}_l(\E)$ under kernels of epics and cokernels of monics.
		\item The subcategory ${}^0{\Fac}^\E_{l+1}(\omega)$ is closed in $\Fac^\E_{l+1}(\omega)$ under extensions.
	\end{enumerate}
	In particular, $\Mor^{\textup{m}_\A}_l(\E)$, $\Fac^\E_{l+1}(\omega)$, and ${}^0{\Fac}^\E_{l+1}(\omega)$ are fully exact subcategories of $\Mor_{l}(\E)$.
\end{lem}

\begin{proof} \
	\begin{enumerate}
		\item This is a direct consequence of the Five lemma, see {\cite[Cor.~3.2]{Buh10}}.
		
		\item By \Cref{lem: Noether}, any short exact sequence $\begin{tikzcd}[cramped, sep=small] (X, \alpha) \ar[r, tail] & (Y, \beta) \ar[r, two heads] & (Z, \gamma) \end{tikzcd}$ in $\Mor^{\textup{m}_\A}_l(\E)$ induces a short exact sequence $\begin{tikzcd}[cramped, sep=small] \cok(\alpha^{l-1} \cdots \alpha^0) \ar[r, tail] & \cok(\beta^{l-1} \cdots \beta^0) \ar[r, two heads] & \cok(\gamma^{l-1} \cdots \gamma^0) \end{tikzcd}$ of cokernels. Due \Cref{lem: A/omega-fully-exact} and \Cref{rmk: Fac-cok-A/omega}, $Y \in \Fac^\E_{l+1}(\omega)$ then implies $X, Z \in \Fac^\E_{l+1}(\omega)$, and the claim follows.
		
		\item This holds, since $\Proj(\A) \cap \E$ is closed in $\E$ under extensions.
	\end{enumerate}
	The particular statement follows from \Cref{lem: fullyexact}.
\end{proof}

\begin{rmk} \label{rmk: circle-zig-zag}
	Let $(\tau, \omega)$ be a twisted homothety on an exact category $\A$, $\E = \tau \E$ a fully exact subcategory, and $l \in \NN$. For any $X=(X, \alpha) \in \Fac^\E_{l+1}(\omega)$, we have $\omega_{X^{k}} = \tau (\alpha^{k-1}\cdots \alpha^{0})\alpha^l\cdots\alpha^{k}$ by postcomposing with the monic $\tau(\alpha^{l-1}\cdots \alpha^{k})$ for any $k \in \{0, \dots, l-1\}$:
	\begin{center}
		\begin{tikzcd}[sep={20mm,between origins}]
			X^{0} \ar[r, tail, "\alpha^{0}"] \ar[d, "\omega_{X^{0}}"] & \cdots \ar[r, tail, "\alpha^{k-1}"] & X^{k} \ar[r, tail, "\alpha^{k}"] & \cdots \ar[r, tail, "\alpha^{l-1}"] & X^l \ar[d, "\omega_{X^{l}}"] \ar[dllll, "\alpha^l"' near end] \\
			\tau X^{0} \ar[r, tail, "\tau \alpha^{0}"'] & \cdots \ar[r, tail, "\tau \alpha^{k-1}"'] & \tau X^{k} \ar[r, tail, "\tau \alpha^{k}"'] \ar[u, <-, "\omega_{X^{k}}"' near end, crossing over] & \cdots \ar[r, tail, "\tau \alpha^{l-1}"'] & \tau X^l 
		\end{tikzcd}
	\end{center}
\end{rmk}

\begin{rmk} \label{rmk: circle-morphisms}
	Let $(\tau, \omega)$ be a twisted homothety on an exact category $\A$, $\E = \tau \E$ a fully exact subcategory of $\A$, and $l \in \NN$. For any morphism $f\colon (X, \alpha) \to (Y, \beta)$ in $\Fac^\E_{l+1}(\omega)$, we obtain $\tau(f^0) \alpha^l = \beta^l f^l$ by postcomposing with the monic $\tau(\beta^{l-1} \cdots \beta^0)$:
	\begin{center}
		\begin{tikzcd}[sep={22.5mm,between origins}]
			X^l \ar[r, "\alpha^l"'] \ar[d, "f^l"] \ar[rr, bend left=7.5mm, "\omega_{X^l}"] & \tau X^0 \ar[r, tail, "\tau(\alpha^{l-1} \cdots \alpha^0)"'] \ar[d, "\tau f^0"] & \tau X^l \ar[d, "\tau f^l"] \\
			Y^l \ar[r, "\beta^l"] \ar[rr, bend right=7.5mm, "\omega_{Y^l}"'] & \tau Y^0 \ar[r, tail, "\tau(\beta^{l-1} \cdots \beta^0)"] & \tau Y^l
		\end{tikzcd}
	\end{center}
	Thus, sending $(X, \alpha) \in \Fac^\E_{l+1}(\omega)$ to the \enquote{helicoidal} sequence
	\begin{center}
		\begin{tikzcd}
			\cdots \ar[r, tail, "\tau^{-1} \alpha^{l-1}"] & \tau^{-1} X^l \ar[r, "\tau^{-1} \alpha^l"] & X^{0} \ar[r, tail, "\alpha^{0}"] & \cdots \ar[r, tail, "\alpha^{l-1}"] & X^l \ar[r, "\alpha^l"] & \tau X^0 \ar[r, tail, "\tau \alpha^0"] & \cdots \ar[r, tail, "\tau \alpha^{l-1}"] & \tau X^l \ar[r, "\tau \alpha^l"] & \cdots
		\end{tikzcd}
	\end{center}
	gives rise to an embedding of $\Fac^\E_{l+1}(\omega)$ into the category of diagrams of type $A_\infty^\infty$ in $\E$, where $A_\infty^\infty$ is the unidirectional linear quiver, infinite on both sides. If $\tau = \id_\A$, sending $(X, \alpha) \in \Fac^\E_{l+1}(\omega)$ to
	\begin{center}
		\begin{tikzcd}[sep=large]
			X^{0} \ar[r, tail, "\alpha^{0}"] & X^{1} \ar[r, tail, "\alpha^{1}"] & \cdots \ar[r, tail, "\alpha^{l-2}"] & X^{l-1} \ar[r, tail, "\alpha^{l-1}"] & X^l, \ar[llll, bend left=7mm, "\alpha^l"]
		\end{tikzcd}
	\end{center}
	gives rise to an embedding of $\Fac^\E_{l+1}(\omega)$ into the category of diagrams of type $C_{l+1}$ in $\E$, where $C_{l+1}$ is the unidirectional circular quiver of size $l+1$.
\end{rmk}

\begin{rmk} \label{rmk: alpha-monic}
	Suppose in \Cref{dfn: Fac} that $\A$ is weakly idempotent complete and $\omega$ regular on $\E$. Then $\alpha^l$ is an admissible monic in $\A$ due to \Cref{prp: WIC}.\ref{prp: WIC-comp-monic}.
\end{rmk}

In view of \Cref{rmk: circle-zig-zag,rmk: alpha-monic,rmk: circle-morphisms}, $\Fac^\E_{l+1}(\omega)$ generalizes the category of matrix factorizations with $l+1$ factors. Under the respective assumptions, it has a circular symmetry, up to a twist by $\tau$.

\begin{dfn}
	Let $(\tau, \omega)$ be a twisted homothety on an exact category $\A$, $\E = \tau \E$ a fully exact subcategory of $\A$, and $l \in \NN$. For any $X=(X, \alpha) \in \Fac^\E_{l+1}(\omega)$, where $\alpha^l$ is a monic in $\A$, we define its \textbf{rotation} and \textbf{reverse rotation} as the objects
	\begin{center}
		\begin{tikzcd}[sep=large]
			\Theta X = (\Theta X, \beta) \colon \; X^1 \ar[r, tail, "\alpha^1"] & X^2 \ar[r, tail, "\alpha^2"] & \cdots \ar[r, tail, "\alpha^{l-1}"] & X^l \ar[r, tail, "\alpha^l"] & \tau X^0
		\end{tikzcd}
	\end{center}
	and
	\begin{center}
		\begin{tikzcd}[sep=large]
			\Theta^{-1} X = (\Theta^{-1} X, \beta) \colon \; \tau^{-1} X^l \ar[r, tail, "\tau^{-1} \alpha^l"] & X^0 \ar[r, tail, "\alpha^0"] & \cdots \ar[r, tail, "\alpha^{l-3}"] & X^{l-2} \ar[r, tail, "\alpha^{l-2}"] & X^{l-1}
		\end{tikzcd}
	\end{center}
	of $\Fac^\E_{l+1}(\omega)$, where $\beta ^l=\tau \alpha^0$ and $\beta ^l=\alpha^{l-1}$, respectively, see \Cref{rmk: circle-zig-zag}. If $\A$ is weakly idempotent complete and $(\tau, \omega)$ regular on $\E$, rotation defines an automorphism of $\Fac^\E_{l+1}(\omega)$ such that $\Theta^{l+1} = \tau$, see \Cref{rmk: circle-morphisms,rmk: alpha-monic}.
\end{dfn}

\begin{dfn}
	Let $(\tau, \omega)$ be a twisted homothety on an exact category $\A$, $\E = \tau \E$ a fully exact subcategory of $\A$, and $l \in \NN$. Then $\mu_{l+1}$ from \Cref{ntn: mMor}.\ref{ntn: mMor-mu} defines a functor 
	\[\nu^l \colon \E \longrightarrow \Fac^\E_{l+1}(\omega), \; A \longmapsto \nu^l(A) = (\nu^l(A), \alpha) := \mu_{l+1}(A),\]
	where $\alpha^l=\omega_A$. Suppose that $(\tau, \omega)$ is regular on $\E$.
	For $k \in \{0, \dots, l-1\}$, any $A \in \E$ gives rise to an object
	\begin{center}
		\begin{tikzcd}
			\nu^k(A)=(\nu^k(A), \alpha)\colon A \ar[r, equal] & \cdots \ar[r, equal] & A \ar[r, tail, "\omega_A"] & \tau A \ar[r, equal] & \cdots \ar[r, equal] & \tau A
		\end{tikzcd}
	\end{center}
	of $\Fac^\E_{l+1}(\omega)$, where $\alpha^{k} = \omega_A$ and $\alpha^l=\id_A$. This defines a functor $\nu^k\colon \E \to \Fac^\E_{l+1}(\omega)$. We refer to the objects $\nu^0(A), \dots, \nu^l(A) \in \Fac^\E_{l+1}(\omega)$ as \textbf{trivial factorizations}.
\end{dfn}

\begin{rmk} \label{rmk: nu-Theta} Let $(\tau, \omega)$ be a twisted homothety on an exact category $\A$, regular on a fully exact subcategory $\E = \tau \E$ of $\A$, $l \in \NN$, and $k \in \{0,\dots,l\}$.
	\begin{enumerate}
		\item \label{rmk: nu-Theta-1} The functor $\nu^k$ is exact, since $\tau$ is so.
		
		\item \label{rmk: nu-Theta-2} We have $\Theta \nu^{k+1} = \nu^k$ for $k < l$ and $\Theta \nu^0 = \nu^l\tau$.
		
		\item \label{rmk: nu-Theta-3} We have $\pi^k \Theta = \pi^{k+1}$ for $k <l$ and $\pi^l \Theta = \tau \pi^0$, whenever application of $\Theta$ is defined.
	\end{enumerate}
\end{rmk}

\begin{lem} \label{lem: nu-adjunction}
	Let $(\tau, \omega)$ be a twisted homothety on an exact category $\A$, $\E = \tau \E$ a fully exact subcategory of $\A$, $l \in \NN$, and $k \in \{0, \dots, l\}$. For statements involving $\nu^k$ for $k<l$, suppose that $(\tau, \omega)$ is regular on $\E$. There are the following adjunctions:
	\begin{enumerate}
		\item \label{lem: nu-adjunction-left-0} $\nu^l \dashv \pi^0$, given by $\Hom_{\Fac^\E_{l+1}(\omega)}(\nu^l(-), -) \cong \Hom_\E(-, \pi^0(-))$, 
		$g \longmapsto g^0$,
		
		\item \label{lem: nu-adjunction-left} 
 				$\nu^{k-1} \dashv \tau^{-1}\pi^k$, for $k > 0$, given by $\Hom_{\Fac^\E_{l+1}(\omega)}(\nu^{k-1}(-), -) \cong \Hom_\E(\tau(-), \pi^k(-))$, 
				$g \longmapsto g^k$,
		
		\item \label{lem: nu-adjunction-right}
				$\pi^k \dashv \nu^k$, given by $\Hom_\E(\pi^k(-), -) \cong \Hom_{\Fac^\E_{l+1}(\omega)}(-, \nu^k(-))$, $g^k \reflectbox{ $\longmapsto$ } g$.
	\end{enumerate}
	In particular, $\nu^k(A)$ is projective, or, injective, in $\Fac^\E_{l+1}(\omega)$ if $A \in \E$ is so.
\end{lem}

\begin{proof}
	Part \ref{lem: nu-adjunction-left-0} is obvious. We only prove \ref{lem: nu-adjunction-left} since \ref{lem: nu-adjunction-right} is similar. For injectivity, consider a morphism
	\begin{center}
		\begin{tikzcd}[sep={17.5mm,between origins}]
			\nu^{k-1}(A) \ar[d, "g"] & A \ar[r, equal] \ar[d, "g^0"] & \cdots \ar[r, equal] & A \ar[r, tail, "\omega_A"] \ar[d, "g^{k-1}"] & \tau A \ar[r, equal] \ar[d, "g^k"] & \tau A \ar[r, equal] \ar[d, "g^{k+1}"] & \cdots \ar[r, equal] & \tau A \ar[d, "g^l"]\\
			X & X^0 \ar[r, tail, "\alpha^{0}"] & \cdots \ar[r, tail, "\alpha^{k-2}"] & X^{k-1} \ar[r, tail, "\alpha^{k-1}"] & X^k \ar[r, tail, "\alpha^{k}"] & X^{k+1} \ar[r, tail, "\alpha^{k+1}"] & \cdots \ar[r, tail, "\alpha^{l-1}"] & X^l,
		\end{tikzcd}
	\end{center}
	in $\Fac^\E_{l+1}(\omega)$, where $A \in \E$, and suppose that $g^k$ is zero. Then so are $g^j = \alpha^{j-1} \cdots \alpha^k g^k$ for $j \in \{k+1,\dots,l\}$. Due to \Cref{rmk: circle-morphisms}, the diagram
	\begin{equation*}
		\begin{tikzcd}[sep={17.5mm,between origins}]
			\tau A \ar[d, "g^l"] \ar[r, equal] & \tau A \ar[r, equal] \ar[d, "\tau g^0"] & \cdots \ar[r, equal] & \tau A \ar[d, "\tau g^{k-1}"] \\
			X^l \ar[r, "\alpha^l"] & \tau X^0 \ar[r, tail, "\tau \alpha^{0}"] & \cdots \ar[r, tail, "\tau \alpha^{k-2}"] & \tau X^{k-1},
		\end{tikzcd}
	\end{equation*}
	commutes, and hence $\tau(g^j)=\tau(\alpha^{j-1} \cdots \alpha^0) \alpha^l g^l =0$ for $j \in \{0, \dots,k-1\}$.
	It follows that $g=0$ since $\tau$ is faithful. For surjectivity, consider $X \in \Fac^\E_{l+1}(\omega)$ and suppose that a morphism $g^k \colon \tau A \to X^k$ in $\E$ is given. Define $g^0 := \tau^{-1}(\alpha^l \cdots \alpha^k g^k)$, and set $g^j := \alpha^{j-1} g^{j-1}$ for $j \in \{1, \dots, k-1, k+1, \dots l\}$. We obtain a diagram
	\begin{center}
		\begin{tikzcd}[sep={17.5mm,between origins}]
			A \ar[rr, tail, equal] \ar[d, "\tau^{-1} g^k"] && A \ar[r, equal] \ar[d, dashed, "g^0"] & \cdots \ar[r, equal] & A \ar[r, tail, "\omega_A"] \ar[d, dashed, "g^{k-1}"] & \tau A \ar[r, equal] \ar[d, "g^k"] & \tau A \ar[r, equal] \ar[d, dashed, "g^{k+1}"] & \cdots \ar[r, equal] & \tau A \ar[d, dashed, "g^l"]\\
			\tau^{-1} X^k \ar[rr, "\tau^{-1}(\alpha^l \cdots \alpha^k)"] \ar[rrrrr, "\omega_{X^k}"', bend right=5mm] && X^0 \ar[r, tail, "\alpha^{0}"] & \cdots \ar[r, tail, "\alpha^{k-2}"] & X^{k-1} \ar[r, tail, "\alpha^{k-1}"] & X^k \ar[r, tail, "\alpha^{k}"] & X^{k+1} \ar[r, tail, "\alpha^{k+1}"] & \cdots \ar[r, tail, "\alpha^{l-1}"] & X^l,
		\end{tikzcd}
	\end{center}
	which commutes due to \Cref{rmk: circle-zig-zag}. Thus, the desired preimage is the morphism
	\[g=(g^0, \dots, g^l) \colon \nu^l(A) \longrightarrow X\]
	in $\Fac^\E_{l+1}(\omega)$. The naturality of this bijection is obvious.\\
	For the particular claim, let $k \in \{0, \dots, l-1\}$. For $A \in \Proj(\E)$, also $\tau A \in \Proj(\E)$, see \Cref{rmk: tau-restricted}. Then \[\Hom_{\Fac^\E_{l+1}(\omega)}(\nu^k(A), -) \cong \Hom_\E(\tau A, \pi^{k+1}(-)) \hspace{5mm} \textup{and} \hspace{5mm} \Hom_{\Fac^\E_{l+1}(\omega)}(\nu^l(A), -) \cong \Hom_\E(A, \pi^0(-))\] are exact functors. This means that $\nu^k(A), \nu^l(A) \in \Proj\left(\Fac^\E_{l+1}(\omega)\right)$. The statement on injectives follows similarly.
\end{proof}

In the module case, the following observation by Sun and Zhang opens the way to their approach by means of Frobenius functors, see {\cite[7]{SZ26}}.

\begin{rmk}
	Let $(\tau, \omega)$ be a twisted homothety on an exact category $\A$, regular on a fully exact subcategory $\E = \tau \E$ of $\A$, $l \in \NN$, and $k \in \{0,\dots,l\}$. Combining \Cref{lem: nu-adjunction} and \Cref{rmk: nu-Theta}.\ref{rmk: nu-Theta-2} and \ref{rmk: nu-Theta-3}, we obtain Frobenius pairs
	\begin{enumerate}
		\item $(\nu^k, \tau^{-1} \pi^{k+1})$ for $k<l$ and $(\nu^l, \pi^0)$ of type $(\tau, \Theta^{-1})$ since $\nu^k \vdash \tau^{-1}\pi^{k+1} \vdash \nu^{k+1} \tau = \Theta^{-1} \nu^k \tau$ for $k < l$ and $\nu^l \vdash \pi^0 \vdash \nu^0=\Theta^{-1} \nu^l \tau$,
		\item $(\pi^k, \nu^k)$ of type $(\Theta, \tau^{-1})$, if $\A$ is weakly idempotent complete, since $\pi^k \vdash \nu^k \vdash \tau^{-1} \pi^{k+1} = \tau^{-1} \pi^k \Theta$ for $k < l$ and $\pi^l \vdash \nu^l \vdash \pi^0 = \tau^{-1} \pi^l \Theta$.
	\end{enumerate}
\end{rmk}

\begin{lem} \label{lem: nu-admissible}
	Let $(\tau, \omega)$ be a twisted homothety on an exact weakly idempotent complete category $\A$, regular on a fully exact subcategory $\E = \tau \E$ of $\A$, and $l \in \NN$. In $\Fac^\E_{l+1}(\omega)$, any object $X=(X,\alpha)$ admits
	\begin{enumerate}
		\item \label{lem: nu-admissible-epic} an admissible epic $\epic{\nu^l(X^0) \oplus \bigoplus_{k=1}^l \nu^{k-1}(\tau^{-1}X^k)}{X}$,
		\item \label{lem: nu-admissible-monic} an admissible monic $\monic{X}{\bigoplus_{k=0}^l \nu^k(X^k)}$.
	\end{enumerate}
\end{lem}

\begin{proof}
	We only prove \ref{lem: nu-admissible-epic} since \ref{lem: nu-admissible-monic} is similar. For $j \in \{0, \dots, l\}$, set
	\[X^{\widehat j} := \bigoplus_{k \in \{0, \dots, j-1\}} X^k \oplus \bigoplus_{k \in \{j+1, \dots, l\}} \tau^{-1}X^k \in \E \hspace{2.5mm} \textup{and} \hspace{2.5mm} \tilde X = (\tilde X, \varphi) := \nu^l(X^0) \oplus \bigoplus_{k=1}^l \nu^{k-1}(\tau^{-1} X^k).\] Under the adjunctions from \Cref{lem: nu-adjunction}.\ref{lem: nu-adjunction-left-0} and \ref{lem: nu-adjunction-left}, the identities $\id_{X^j}$ correspond to morphisms $\nu^l(X^0) \to X$ and $\nu^{j-1}(\tau^{-1}X^j) \to X$ in $\Fac^\E_{l+1}(\omega)$ for $j>0$. Combined, these form a termwise split admissible epic $\epic{\tilde X}{X}$ in $\Mor_l(\E)$ with kernel $\widehat X = (\widehat X, \psi)$ as follows:
	\begin{center}
		\begin{tikzcd}[sep={25mm,between origins}, ampersand replacement=\&]
			\widehat X \mathrlap \colon \ar[d, tail] \& \cdots \ar[r] \& X^{\widehat j} \ar[r, "\psi^j"] \ar[tail]{d}{\begin{pmatrix} - \beta_j \\ \id_{X^{\widehat j}} \end{pmatrix}} \& X^{\widehat{j+1}} \ar[r] \ar[tail]{d}{\begin{pmatrix} - \beta_{j+1} \\ \id_{X^{\widehat{j+1}}} \end{pmatrix}} \& \cdots \\
			\tilde X \mathrlap \colon \ar[d, two heads] \& \cdots \ar[r, tail] \& X^j \oplus X^{\widehat j} \ar[r, tail, "\varphi^j"] \ar[two heads]{d}{\begin{pmatrix} \id_{X^j} & \beta_j \end{pmatrix}} \& X^{j+1} \oplus X^{\widehat{j+1}} \ar[r, tail] \ar[two heads]{d}{\begin{pmatrix} \id_{X^{j+1}} & \beta_{j+1} \end{pmatrix}} \& \cdots \\
			X \mathrlap \colon \& \cdots \ar[r, tail] \& X^{j} \ar[r, tail, "\alpha^j"] \& X^{j+1} \ar[r, tail] \& \cdots,
		\end{tikzcd}
	\end{center}
	where
	\[\varphi^j = \left(\begin{array}{c|ccccccc}
		0 & 0 & \cdots & 0 & \omega_{\tau^{-1}X^{j+1}} & 0 & \cdots & 0 \\ \hline
		0 & \id_{X^0} & \cdots & 0 & 0 & 0 & \cdots & 0 \\
		\vdots & \vdots & \ddots & \vdots & \vdots & \vdots & \ddots & \vdots \\
		0 & 0 & \cdots & \id_{X^{j-1}} & 0 & 0 & \cdots & 0 \\
		\id_{X^j} & 0 & \cdots & 0 & 0 & 0 & \cdots & 0 \\
		0 & 0 & \cdots & 0 & 0 & \id_{\tau^{-1} X^{j+2}} & \cdots & 0 \\
		\vdots & \vdots & \ddots & \vdots & \vdots & \vdots & \ddots & \vdots \\
		0 & 0 & \cdots & 0 & 0 & 0 & \cdots & \id_{\tau^{-1} X^l}
	\end{array}\right)\] 
	and
	\[\beta_j := \begin{pmatrix}
		\alpha^{j-1} \cdots \alpha^0 & \cdots & \alpha^{j-1} & \alpha^{j-1} \cdots \alpha^0 \tau^{-1}(\alpha^l \cdots \alpha^{j+1}) & \cdots & \alpha^{j-1} \cdots \alpha^0 \tau^{-1}(\alpha^l)
	\end{pmatrix}.\]
	It remains to see that $\widehat X$ lies in $\Fac^\E_{l+1}(\omega)$. For any $j$, using the left-inverse $\begin{pmatrix} 0 & \id_{X^{\widehat {j+1}}} \end{pmatrix}$ of $\begin{pmatrix} - \beta_{j+1} \\ \id_{X^{\widehat{j+1}}} \end{pmatrix}$,
	
		\begin{align*}
			& \psi^j = \begin{pmatrix} 0 & \id_{X^{\widehat {j+1}}} \end{pmatrix} \varphi^j \begin{pmatrix} - \beta_j \\ \id_{X^{\widehat j}} \end{pmatrix} = \\
			& \begin{scriptsize}\begin{pmatrix}
				\id_{X^0} & \cdots & 0 & 0 & 0 & \cdots & 0 \\
				\vdots & \ddots & \vdots & \vdots & \vdots & \ddots & \vdots \\
				0 & \cdots & \id_{X^{j-1}} & 0 & 0 & \cdots & 0 \\
				-\alpha^{j-1} \cdots \alpha^0 & \cdots & -\alpha^{j-1} & -\alpha^{j-1} \cdots \alpha^0 \tau^{-1}(\alpha^l \cdots \alpha^{j+1}) & -\alpha^{j-1} \cdots \alpha^0 \tau^{-1}(\alpha^l \cdots \alpha^{j+2}) & \cdots & -\alpha^{j-1} \cdots \alpha^0 \tau^{-1}(\alpha^l) \\
				0 & \cdots & 0 & 0 & \id_{\tau^{-1} X^{j+2}} & \cdots & 0 \\
				\vdots & \ddots & \vdots & \vdots & \vdots & \ddots & \vdots \\
				0 & \cdots & 0 & 0 & 0 & \cdots & \id_{\tau^{-1} X^l}
			\end{pmatrix}\end{scriptsize}.
		\end{align*}
	
	This is isomorphic to the admissible monic $\monic{X^{\widehat j}}{X^{\widehat {j+1}}}$ in $\A$ given by
	\[\left(\begin{array}{ccccccc}
		\id_{X^0} & \cdots & 0 & 0 & 0 & \cdots & 0 \\
		\vdots & \ddots & \vdots & \vdots & \vdots & \ddots & \vdots \\
		0 & \cdots & \id_{X^{j-1}} & 0 & 0 & \cdots & 0 \\
		0 & \cdots & 0 & -\alpha^{j-1} \cdots \alpha^0 \tau^{-1}(\alpha^l \cdots \alpha^{j+1}) & 0 & \cdots & 0 \\
		0 & \cdots & 0 & 0 & \id_{\tau^{-1} X^{j+2}} & \cdots & 0 \\
		\vdots & \ddots & \vdots & \vdots & \vdots & \ddots & \vdots \\
		0 & \cdots & 0 & 0 & 0 & \cdots & \id_{\tau^{-1} X^l}
	\end{array}\right),\]
	see \Cref{prp: Buehler2.9} and \Cref{rmk: alpha-monic}, and $\psi^j$ itself is an admissible monic in $\A$. This means that $\widehat X$ lies in $\Mor^{\textup{m}_\A}_l(\E)$, and thus in $\Fac^\E_{l+1}(\omega)$ due to \Cref{lem: Fact-fully-exact}.\ref{lem: Fact-fully-exact-co-kernels}.
\end{proof}

\begin{prp} \label{prp: Fact-enough}
	Let $(\tau, \omega)$ be a twisted homothety on an exact weakly idempotent complete category $\A$, regular on a fully exact subcategory $\E = \tau \E$ of $\A$, and $l \in \NN$. 
	\begin{enumerate}
		\item \label{prp: Fact-enough-Frob} Suppose that $\E$ has enough projectives. Then $\Fac^\E_{l+1}(\omega)$ has enough projectives. These are the direct summands of direct sums of objects of the form $\nu^k(P) \in \Fac^\E_{l+1}(\omega)$, where $P \in \Proj(\E)$ and $k \in \{0, \dots, l\}$. The same statements hold verbatim for injectives.
		
		\item \label{prp: Fact-enough-subFrob} Suppose that $\E$ is Frobenius. Then $\Fac^\E_{l+1}(\omega)$ is a Frobenius category. If $\E$ has enough $\A$-projectives, then ${}^0{\Fac}^\E_{l+1}(\omega)$ a sub-Frobenius category with the same projective-injectives.
		In particular, there is a fully faithful triangle functor ${}^0{\ul{\Fac}}^\E_{l+1}(\omega) \hookrightarrow \ul\Fac^\E_{l+1}(\omega)$.
	\end{enumerate}
\end{prp}

\begin{proof}\
	\begin{enumerate}
		\item Consider $X \in \Fac^\E_{l+1}(\omega)$ and admissible epics $\epic{P^k}{X^k}$ in $\E$ with $P^k \in \Proj(\E)$ for $k \in \{0,\dots,l\}$. By \Cref{rmk: nu-Theta}.\ref{rmk: nu-Theta-1} and \Cref{prp: Buehler2.9}, these form an admissible epic
		\[\Proj \left(\Fac^\E_{l+1}(\omega)\right) \ni P(X) := \epic{\nu^l(P^0) \oplus \bigoplus_{k=1}^l \nu^{k-1}(\tau^{-1}P^k)}{\nu^l(X^0) \oplus \bigoplus_{k=1}^l \nu^{k-1}(\tau^{-1}X^k)}\] in $\Fac^\E_{l+1}(\omega)$, see the particular statement of \Cref{lem: nu-adjunction}. Composing with the admissible epic from \Cref{lem: nu-admissible}.\ref{lem: nu-admissible-epic} yields an admissible epic $p_X \colon \epic{P(X)}{X}$ in $\Fac^\E_{l+1}(\omega)$ and the claims on projectives follow.
		Dually, we can construct an admissible monic
		\[i_X \colon \monic{X}{\bigoplus_{k=0}^l \nu^k(I^k) =: I(X) \in \Inj\left(\Fac^\E_{l+1}(\omega)\right)}\]
		in $\Fac^\E_{l+1}(\omega)$ from admissible monics $\monic{X^k}{I^k}$ in $\E$, where $I^k \in \Inj(\E)$. The claims on injectives then follow in the same way.
		
		\item The first claim follows immediately from \ref{prp: Fact-enough-Frob}. If $\E$ has enough $\A$-projectives, then $\Inj(\E)=\Proj(\E)=\Proj(\A) \cap \E$, see \Cref{lem: sub-Frobenius}.\ref{lem: sub-Frobenius-a}. In the proof of \ref{prp: Fact-enough-Frob}, then $P(X), I(X) \in {}^0{\Fac}^\E_{l+1}(\omega)$. In addition, the kernel $K$ of $p_X$ and the cokernel $C$ of $i_X$ lie in ${}^0{\Fac}^\E_{l+1}(\omega)$: Indeed, for $X \in {}^0{\Fac}^\E_{l+1}(\omega)$, both short exact sequences \[\begin{tikzcd}[cramped, sep=small] K^l \ar[r, tail] & P(X)^l \ar[r, two heads] & X^l \end{tikzcd} \hspace{5mm} \textup{and} \hspace{5mm} \begin{tikzcd}[cramped, sep=small] X^l \ar[r, tail] & I(X)^l \ar[r, two heads] & C^l \end{tikzcd}\]
		in $\E$ split due to $X^l \in \Proj(\A) \cap \E = \Proj(\E) \cap \Inj(\E)$. Hence, $K^l, C^l \in \Proj(\E)$ as direct summands of $P(X)^l, I(X)^l \in \Proj(\E)$, respectively. It follows that $p_X$ and $i_X$ are admissible in ${}^0{\Fac}^\E_{l+1}(\omega)$. Therefore, ${}^0{\Fac}^\E_{l+1}(\omega)$ has enough $\Fac^\E_{l+1}(\omega)$-projectives and -injectives. Since ${}^0{\Fac}^\E_{l+1}(\omega)$ is closed in ${\Fac}^\E_{l+1}(\omega)$ under direct summands, the projective-injectives of $\Fac^\E_{l+1}(\omega)$ lie in ${}^0{\Fac}^\E_{l+1}(\omega)$. Thus, all remaining claims are due to \Cref{lem: sub-Frobenius}. \qedhere
	\end{enumerate}
\end{proof}

\section{Equivalence with comma categories}

In this section, we describe the category of factorizations over an exact category as a comma category and as a diagram category. The established exact equivalences express the fact that any factorization can be reconstructed from its (generalized) cokernel and the contraction of its monics.

\begin{con} \label{con: coker}
	Let $(\tau, \omega)$ be a twisted homothety on an exact category $\A$, $\E = \tau \E$ a fully exact subcategory of $\A$, and $l \in \NN$. On any $X=(X, \alpha) \in \Fac^\E_{l+1}(\omega)$, the counit $j\colon \nu^l \pi^0 \to \id_{\Fac^\E_{l+1}(\omega)}$ of the adjunction from \Cref{lem: nu-adjunction}.\ref{lem: nu-adjunction-left-0} fits into a commutative diagram of bicartesian squares
	\begin{align} \label{diag: coker}
		\begin{tikzcd}[sep={17.5mm,between origins}, ampersand replacement=\&]
			\nu^l(X^{0}) \ar[d, tail, "j_X"] \& X^{0} \ar[r, equal] \ar[d, equal] \& X^{0} \ar[r, equal] \ar[d, tail, "\alpha^0"] \& X^{0} \ar[r, equal] \ar[d, tail, "\alpha^1 \alpha^0"] \& \cdots \ar[r, equal] \& X^{0} \ar[r, equal] \ar[d, tail, "\alpha^{l-2} \cdots \alpha^0"] \& X^{0} \ar[d, tail, "\alpha^{l-1} \cdots \alpha^0"]\\
			X \ar[d, two heads, "q_X"] \& X^{0} \ar[r, tail, "\alpha^{0}"] \ar[d, two heads] \ar[rd, phantom, "\square"] \& X^{1} \ar[r, tail, "\alpha^{1}"] \ar[d, two heads] \ar[rd, phantom, "\square"] \& X^{2} \ar[r, tail, "\alpha^{2}"] \ar[d, two heads] \& \cdots \ar[r, tail, "\alpha^{l-2}"] \ar[rd, phantom, "\square"] \& X^{l-1} \ar[r, tail, "\alpha^{l-1}"] \ar[d, two heads] \ar[rd, phantom, "\square"] \& X^l \ar[d, two heads] \\
			U_X\&0 \ar[r, tail] \& U_X^{1} \ar[r, tail] \& U_X^{2} \ar[r, tail] \ar[ru, phantom, "\square"]  \& \cdots \ar[r, tail] \& U_X^{l-1} \ar[r, tail] \& U_X^l,
		\end{tikzcd}
	\end{align}
	obtained by successive pushouts in $\A$, see \Cref{prp: Buehler2.12}.\ref{prp: Buehler2.12-push}. It represents a short exact sequence in $\mMor_{l}(\A)$. Due to \Cref{prp: Buehler2.12}.\ref{prp: Buehler2.12-pull} and \Cref{rmk: Fac-cok-A/omega}, we have
	\begin{gather} \label{eqn: coker}
		\cok(U_X^{j} \rightarrowtail U_X^{k}) \cong \cok(X^{j} \rightarrowtail X^{k}) \in \A/\omega
	\end{gather}
	for all $j, k \in \{0, \dots, l\}$ with $j < k$. We call the object $\cok_l(X) = \cok(X)$ of $\mMor_{l-1}(\A/\omega)$ defined by
	\begin{gather} \label{eqn: U-iota-coker}
		U_X = \iota\left(\cok(X)\right)
	\end{gather}
	the \textbf{cokernel} of $X$, see \Cref{ntn: mMor}.\ref{ntn: mMor-iota} and \Cref{lem: A/omega-fully-exact}.
\end{con}

\begin{lem} \label{lem: coker-exact}
	Let $(\tau, \omega)$ be a twisted homothety on an exact category $\A$, $\E = \tau \E$ a fully exact subcategory of $\A$, and $l \in \NN$. Then \Cref{con: coker} defines an exact functor
	\[\Fac^\E_{l+1}(\omega) \xrightarrow{\cok_l \, = \, \cok} \mMor_{l-1}(\A/\omega),\]
	which fits into a (component-wise) short exact sequence
	\begin{center}
		\begin{tikzcd}[sep={25mm,between origins}]
			\nu^l \circ \pi^0 \ar[r, tail, "j_l \, = \, j"] & \id \ar[r, two heads, "q_l \, = \, q"] & \iota_{l-1} \circ \cok_l
		\end{tikzcd}
	\end{center}
	of exact endofunctors of $\Fac^\E_{l+1}(\omega)$.
\end{lem}

\begin{proof}
	Functoriality of $\cok$ and compatibility with biproducts are obvious. For exactness, apply the natural transformation $j \colon \nu^l \pi^0 \to \id$ of exact functors, see \Cref{ntn: mMor}.\ref{ntn: mMor-pi} and \Cref{rmk: nu-Theta}.\ref{rmk: nu-Theta-1}, to a short exact sequence $\begin{tikzcd}[cramped, sep=small]
		X \ar[r, tail] & Y \ar[r, two heads] & Z
	\end{tikzcd}$ in $\Fac^\E_{l+1}(\omega)$. Due to \eqref{diag: coker}, \eqref{eqn: U-iota-coker}, and \Cref{lem: Noether}, this yields a commutative diagram 
	\begin{center}
		\begin{tikzcd}[row sep={17.5mm,between origins}, column sep={25mm,between origins}]
			\nu^l(X^0) \ar[d, tail, "j_X"] \ar[r, tail] & \nu^l(Y^0) \ar[d, tail, "j_Y"] \ar[r, two heads] & \nu^l(Z^0) \ar[d, tail, "j_Z"] \\
			X \ar[r, tail] \ar[d, two heads, "q_X"] & Y \ar[r, two heads] \ar[d, two heads, "q_Y"] & Z \ar[d, two heads, "q_Z"] \\
			\iota(\cok(X)) \ar[r, tail, dashed] & \iota(\cok(X)) \ar[r, two heads, dashed] & \iota(\cok(X)),
		\end{tikzcd}
	\end{center}
	in $\mMor_{l}(\A)$ with short exact rows and columns. Since $\iota$ reflects exactness, there is a short exact sequence $\begin{tikzcd}[cramped, sep=small]
		\cok(X) \ar[r, tail] & \cok(Y) \ar[r, two heads] & \cok(Z)
	\end{tikzcd}$ in $\mMor_{l-1}(\A/\omega)$ as desired.
\end{proof}

\begin{rmk}
	Suppose that $\E=\A$ in \Cref{con: coker}.
	\begin{enumerate}
		\item There is a torsion pair $(\nu^l(\A), \iota \mMor_{l-1}(\A/\omega))$ of $\Fac^\A_{l+1}(\omega)$.
		\item There is an adjunction $\cok \dashv \iota$ whose unit is $q$.
	\end{enumerate}
	 
\end{rmk}

\begin{lem} \label{lem: coker-E_omega}
	Let $(\tau, \omega)$ be a twisted homothety on an exact category $\A$ and $l \in \NN$. Consider fully exact subcategories $\E = \tau \E$ and $\E_\omega$ of $\A$ and $\A/\omega$, respectively. Suppose that \Cref{asn}.\ref{asn-coker} holds. Then $\cok$ defines an exact functor $\Fac^\E_{l+1}(\omega) \to \mMor_{l-1}(\E_\omega)$.
\end{lem}

\begin{proof}
	We prove that $\cok(X) \in \mMor_{l-1}(\E_\omega)$ in the situation of \Cref{con: coker}. Then the exactness of the functor is due to \Cref{lem: coker-exact}. For any $j, k \in \{0, \dots, l\}$ with $j < k$, \eqref{eqn: coker} yields a short exact sequence
	\begin{center}
		\begin{tikzcd}
			X^{j} \ar[r, tail] & X^{k} \ar[r, two heads] & \cok(U_X^{j} \rightarrowtail U_X^{k}),
		\end{tikzcd} 
	\end{center}
	where $X^j, X^k \in \E$ and $\cok(U_X^{j} \rightarrowtail U_X^{k}) \in \E_\omega$ by \Cref{asn}.\ref{asn-coker}. For $j=0$, this means that $U_X^k \in \E_\omega$, for $k=j+1$, that the monics in $\cok(X)$ are admissible in $\E_\omega$. It follows that $\cok(X) \in \mMor_{l-1}(\E_\omega)$.
\end{proof}

\begin{dfn}[{\cite[\S~II.6]{Mac98}}]
	Given two functors $F \colon \A \to \C$ and $G \colon \B \to \C$, the \textbf{comma category} $(F \downarrow G)$ has as objects triples $(A, \varphi, B)$, where $A \in \A$, $B \in \B$, and $\varphi \colon F A \to G B$ is a morphism in $\C$. Its morphisms $(a, b) \colon (A, \varphi, B) \to (A', \varphi', B')$ are pairs of morphisms $a \colon A \to A'$ in $\A$ and $b \colon B \to B'$ in $\B$ such that
	\begin{center}
		\begin{tikzcd}[sep={17.5mm,between origins}]
			FA \ar[r, "\varphi"] \ar[d, "Fa"] & GB \ar[d, "Gb"] \\
			FA' \ar[r, "\varphi'"] & GB'
		\end{tikzcd}
	\end{center}
	commutes. It is well-known that the comma category is a pullback in the category of (additive) categories:
	\begin{gather} \label{lem: L-diagram-PB}
		\begin{tikzcd}[row sep={17.5mm,between origins}, column sep={25mm,between origins}, ampersand replacement=\&]
			(F \downarrow G) \ar[r] \ar[d] \& \Mor_1(\C) \ar[d]\\
			\A \times \B \ar[r, "F \times G"] \& \C \times \C
		\end{tikzcd}
		\begin{tikzcd}[row sep={17.5mm,between origins}, column sep={35mm,between origins}, ampersand replacement=\&]
			(A, \varphi, B) \ar[r, |->] \ar[d, |->] \& (\varphi \colon FA \to GB) \ar[d, |->] \\
			(A, B) \ar[r, |->] \& (FA, GB)
		\end{tikzcd}
	\end{gather}
\end{dfn}

\begin{ntn}
	Let $(\tau, \omega)$ be a twisted homothety on an exact category $\A$, $\E = \tau \E$ and $\E_\omega$ fully exact subcategories of $\A$ and $\A/\omega$, respectively, and $l \in \NN$.
	\begin{enumerate}
		\item Consider the two functors $\cok_1 \colon \Fac^\E_2(\omega) \to \A/\omega$ and $\pi^{l-1} \colon \mMor_{l-1}(\E_\omega) \to \E_\omega$. Let $\C^\E_{l+1}(\omega)$ denote the subcategory of the comma category $(\cok_1 \downarrow \pi^{l-1})$ consisting of triples $(\tilde X, \varphi, U)$ where $\varphi \colon \cok_1(\tilde X) \cong \pi^{l-1}(U)$ is an isomorphism in $\A$. In this context, we use an indexing where
		\[\tilde X \colon \monic{X^0}{X^l} \hspace{5mm} \textup{ and } \hspace{5mm} \begin{tikzcd}
			U \colon U^1 \ar[r, tail] & \cdots \ar[r, tail] & U^l,
		\end{tikzcd}\] 
		for compatibility with the notation of \Cref{con: coker}. By ${}^0{\C}^\E_{l+1}(\omega)$, we denote the subcategory of $\C^\E_{l+1}(\omega)$ defined by the additional condition $X^l \in \Proj(\A)$.

		\item Let $\L^\E_{l+1}(\omega)$ be the category of diagrams
		\begin{equation} \label{lem: L-diagram}
			\begin{tikzcd}[sep={17.5mm,between origins}, ampersand replacement=\&]
				\& \& \& \& X^0 \ar[d, tail, "\iota"] \\
				\& \& \& \& X^l \ar[d, two heads, "\rho"] \\
				U^{1} \ar[r, tail] \& U^{2} \ar[r, tail] \& \cdots \ar[r, tail] \& U^{l-1} \ar[r, tail] \& U^l
			\end{tikzcd}
		\end{equation}
		in $\A$, where $U \in \mMor_{l-1}(\E_\omega)$, $X^0, X^l\in \E$, and $(\iota, \rho)$ is a short exact sequence. By ${}^0{\L}^\E_{l+1}(\omega)$ we denote its subcategory defined by the additional condition $X^l \in \Proj(\A)$. A diagram like \eqref{lem: L-diagram} is denoted by $((\iota, \rho), U)$. We write a morphism $((\iota, \rho), U) \to ((\kappa, \sigma), V)$ of such diagrams as $(\tilde f,g)$, where $\tilde f=(f^0, f^l) \colon \iota \to \kappa$ and $g=(g^1, \dots, g^l) \colon U \to V$.
	\end{enumerate}
\end{ntn}

\begin{lem}
	Let $(\tau, \omega)$ be a twisted homothety on an exact category $\A$, $\E = \tau \E$ and $\E_\omega$ fully exact subcategories of $\A$ and $\A/\omega$, respectively, and $l \in \NN$. Then ${}^0{\L}^\E_{l+1}(\omega)$ is extension-closed, and hence fully exact in $\L^\E_{l+1}(\omega)$. The category $\L^\E_{l+1}(\omega)$, in turn, is a fully exact subcategory of the category of diagrams in $\A$ of the following type:
	\begin{equation*}
		\begin{tikzcd}[sep={17.5mm,between origins}, ampersand replacement=\&]
			\& \& \& \& \bullet \ar[d] \\
			\& \& \& \& \bullet \ar[d] \\
			\underset 1 \bullet \ar[r] \& \underset 2 \bullet \ar[r] \& \cdots \ar[r] \& \underset {l-1} \bullet \ar[r] \& \underset l \bullet
		\end{tikzcd}
	\end{equation*}
\end{lem}

\begin{proof}
	The category $\mMor_{l-1}(\E_\omega)$ is fully exact in $\Mor_{l-1}(\E_\omega)$ by \Cref{thm: mMor}.\ref{thm: mMor-exact} and hence in $\Mor_{l-1}(\A)$, since $\E_\omega$ is fully exact in $\A$. Due to {\cite[Ex.~3.9]{Buh10}}, the category of short exact sequences $\begin{tikzcd}[cramped, sep=small] X \ar[r, tail] & Y \ar[r, two heads] & Z \end{tikzcd}$ with $X, Y \in \E$ and $Z \in \E_\omega$ is fully exact in $\Mor_2(\A)$, since $\E$ and $\E_\omega$ are so in $\A$. Due to the termwise exact structure of diagram categories, this yields the claim on $\L^\E_{l+1}(\omega)$. The claim on ${}^0{\L}^\E_{l+1}(\omega)$ follows from \Cref{lem: fullyexact}.\ref{lem: fullyexact-extclosed} since $\Proj(\A)\cap \E$ is extension-closed in~$\E$.
\end{proof}

\begin{rmk} \label{rmk: coker-gamma}
	Let $(\tau, \omega)$ be a twisted homothety on an exact category $\A$, $\E = \tau \E$ a fully exact subcategory of $\A$, and $l \in \NN$. There is an isomorphism $\varphi \colon \cok_1 \circ \gamma \cong \pi^{l-1} \circ \cok_l$ of functors $\Fac^\E_{l+1}(\omega) \to \A/\omega$ which fits into the commutative diagram
	\begin{center}
		\begin{tikzcd}[row sep={15mm,between origins}, column sep={27.5mm,between origins}]
		& \pi^1 \circ \gamma \ar[r, equal] \ar[ld, "\pi^1 \circ q_1 \circ \gamma"'] & \pi^l \ar[rd, "\pi^l \circ q_l"] & \\
		\pi^1 \circ \iota_0 \circ \cok_1 \circ \gamma \ar[r, equal] & \cok_1 \circ \gamma \ar[r, dashed, "\varphi", "\cong"'] & \pi^{l-1} \circ \cok_l \ar[r, equal] & \pi^l \circ \iota_{l-1} \circ \cok_l.
		\end{tikzcd}
	\end{center}
\end{rmk}

\begin{lem} \label{lem: comma-embedding}
	Let $(\tau, \omega)$ be a twisted homothety on an exact category $\A$, $\E = \tau \E$ and $\E_\omega$ fully exact subcategories of $\A$ and $\A/\omega$, respectively, and $l \in \NN$. Suppose that \Cref{asn}.\ref{asn-coker} holds. There are the following faithful functors:
	\begin{enumerate}
		\item $\Fac^\E_{l+1}(\omega) \to \C^\E_{l+1}(\omega)$, which sends $X $ to $(\gamma X, \varphi_X, \cok_l(X))$ for any isomorphism $\varphi \colon \cok_1 \circ \gamma \cong \pi^{l-1} \circ \cok_l$ of functors $\Fac^\E_{l+1}(\omega) \to \E_\omega$ as in \Cref{rmk: coker-gamma}, see \Cref{lem: coker-E_omega},
		\item \label{prp: comma-equiv-diag} $\C^\E_{l+1}(\omega) \to \L^\E_{l+1}(\omega)$, which sends $(\tilde X, \varphi, U)$ to $((j_{\tilde X}^1, \varphi q_{\tilde X}^1),U )$, see \eqref{diag: coker}.
	\end{enumerate}
	Their composition $\Fac^\E_{l+1}(\omega) \to \L^\E_{l+1}(\omega)$ sends $X$ to $((j_X^l, q_X^l), \cok_l(X))$, and a morphism $f \colon X \to Y$ to $(\gamma f, \cok_l(f))$. All these statements persist when replacing $\Fac^\E_{l+1}(\omega)$, $\C^\E_{l+1}(\omega)$, and $\L^\E_{l+1}(\omega)$ by their subcategories ${}^0{\Fac}^\E_{l+1}(\omega)$, ${}^0{\C}^\E_{l+1}(\omega)$, and ${}^0{\L}^\E_{l+1}(\omega)$, respectively.
\end{lem}

\begin{proof} \
	\begin{enumerate}
		\item Using the pullback \eqref{lem: L-diagram-PB}, the functors \[\varphi \colon \Fac^\E_{l+1}(\omega) \to \Mor_1(\E_\omega), \; X \mapsto (\varphi_X \colon \cok_1(\gamma X) \to \pi^{l-1}(\cok_l(X))),\] and \[(\gamma, \cok_l) \colon \Fac^\E_{l+1}(\omega) \to \Fac^\E_2(\omega) \times \mMor_{l-1}(\E_\omega), \; X \mapsto (\gamma X, \cok_l(X)),\] induce the desired functor as follows:
		\begin{center}
			\begin{tikzcd}[row sep={17.5mm,between origins}, column sep={35mm,between origins}]
				\Fac^\E_{l+1}(\omega) \ar[rd, dashed] \ar[rrd, "\varphi"] \ar[rdd, "(\gamma{,} \, \cok_l)"'] \\
				& (\cok_1 \downarrow \pi^{l-1}) \ar[r] \ar[d] & \Mor_1(\E_\omega) \ar[d]\\
				& \Fac^\E_2(\omega) \times \mMor_{l-1}(\E_\omega) \ar[r, "\cok_1 \times \pi^{l-1}"] & \E_\omega \times \E_\omega
			\end{tikzcd}
		\end{center}
		Since $\varphi_X$ is an isomorphism for each $X \in \Fac^\E_{l+1}(\omega)$, its image lies in the subcategory $\C^\E_{l+1}(\omega)$ of $(\cok_1 \downarrow \pi^{l-1})$. It is faithful, since $\gamma$ is so and a morphism $f \colon X \to Y$ in $\Fac^\E_{l+1}(\omega)$ is assigned to $(\gamma f, \cok_l(f))$.
		
		\item By definition, any morphism $(\tilde f, g) \colon (\tilde X, \varphi, U) \to (\tilde Y, \psi, V)$ in $\C^\E_{l+1}(\omega)$ yields a commutative diagram
		\begin{equation*}
			\begin{tikzcd}[sep={10mm,between origins}, ampersand replacement=\&]
				\& \&\& \&\& \&\& \&\& \&\& Y^0 \ar[dd, tail] \\
				\& \&\& \&\& \&\& \&\& X^0 \ar[dd, tail] \ar[rru, "f^0"] \\
				\& \&\& \&\& \&\& \&\& \&\& Y^l \ar[dd, two heads, "q^1_{\tilde Y}"] \\
				\& \&\& \&\& \&\& \&\& X^l \ar[dd, two heads, "q^1_{\tilde X}"'] \ar[rru, "f^l"] \\
				\& \&\& \&\& \&\& \&\& \&\& \cok_1(\tilde Y) \ar[dd, "\psi", "\cong"'] \\
				\& \&\& \&\& \&\& \&\& \cok_1(\tilde X) \ar[rru, "\cok_1(\tilde f)", pos=1] \\
				\&\& V^{1} \ar[rr, tail] \&\& V^{2} \ar[rr, tail] \&\& \cdots \ar[rr, tail] \&\& V^{l-1} \ar[rrr, tail] \&\&\& V^l\\
				U^{1} \ar[rr, tail] \ar[rru, "g^1", pos=0.3] \&\& U^{2} \ar[rr, tail] \ar[rru, "g^2", pos=0.3] \&\& \cdots \ar[rr, tail] \&\& U^{l-1} \ar[rrr, tail] \ar[rru, "g^{l-1}", pos=0.2] \&\&\& U^l \ar[rru, "g^l", pos=0.3]
				\arrow[from=6-10, to=8-10, crossing over, "\varphi"', "\cong", pos=0.2] 
			\end{tikzcd}
		\end{equation*}
		in $\A$. Hence, $(\tilde f,g)$ is a morphism $((j_{\tilde X}^1, \varphi q_{\tilde X}^1), U) \to ((j_{\tilde Y}^1, \psi q_{\tilde Y}^1), V)$ in $\L^\E_{l+1}(\omega)$. The identical assignment $(\tilde f,g) \mapsto (\tilde f,g)$ is trivially injective.
	\end{enumerate}
	The composition $\Fac^\E_{l+1}(\omega) \to \L^\E_{l+1}(\omega)$ of these two functors sends an object $X \in \Fac^\E_{l+1}(\omega)$ to \[((j_{\gamma X}^1, \varphi_X q_{\gamma X}^1), \cok_l(X)) = ((j_X^l, q_X^l), \cok_l(X)) \in \L^\E_{l+1}(\omega),\] see \Cref{rmk: coker-gamma}. \\
	All of the previous arguments are unaffected by imposing the condition $X^l \in \Proj(\A)$ on all occurring objects. The claims on the restricted functors follow.
\end{proof}

\begin{ntn} \label{ntn: L}
	Let $(\tau, \omega)$ be a twisted homothety on an exact category $\A$ and $l \in \NN$. Consider fully exact subcategories $\E = \tau \E$ and $\E_\omega$ of $\A$ and $\A/\omega$, respectively. Suppose that \Cref{asn}.\ref{asn-coker} holds. Let $L$ denote the exact functor $\L^\E_{l+1}(\omega) \to \mMor_{l-1}(\E_\omega)$, which sends an object $((\iota,\rho), U)$ to $U$ and a morphism $(\tilde f,g) \colon ((\iota,\rho), U) \to ((\kappa,\sigma), V)$ to $g \colon U \to V$. With the composite functors from \Cref{lem: comma-embedding}, it fits into a commutative diagram
	\begin{equation*}
		\begin{tikzcd}[sep={20mm,between origins}]
			{}^0{\L}_{l+1}^\E(\omega) \ar[rr, hook] \ar[rd, "L"] && \L_{l+1}^\E(\omega) \ar[ld, "L"'] \\
			& \mMor_{l-1}(\E_\omega)\\
			{}^0{\Fac}^\E_{l+1}(\omega) \ar[rr, hook] \ar[uu] \ar[ru, "\cok"] && \Fac^\E_{l+1}(\omega), \ar[uu] \ar[lu, "\cok"']
		\end{tikzcd}
	\end{equation*}
	where restrictions are denoted by the same symbol.
\end{ntn}

\begin{prp} \label{prp: comma-equiv}
	Let $(\tau, \omega)$ be a twisted homothety on an exact category $\A$, $\E = \tau \E$ and $\E_\omega$ fully exact subcategories of $\A$ and $\A/\omega$, respectively, and $l \in \NN$. Suppose that \Cref{asn}.\ref{asn-coker} holds.
	\begin{enumerate}
		\item \label{prp: comma-equiv-a} The functors $\inj{\Fac^\E_{l+1}(\omega)} {\C^\E_{l+1}(\omega)}$ and $\inj{\C^\E_{l+1}(\omega)}{\L^\E_{l+1}(\omega)}$ from \Cref{lem: comma-embedding} are fully faithful.
		
		\item \label{prp: comma-equiv-b} Their composition $\inj{\Fac^\E_{l+1}(\omega)}{\L^\E_{l+1}(\omega)}$ is fully exact and extension-injective. Under \Cref{asn}.\ref{asn-ker}, it is even essentially surjective, and hence an equivalence of categories.
	\end{enumerate}
	All these statements persist when replacing $\Fac^\E_{l+1}(\omega)$, $\C^\E_{l+1}(\omega)$, and $\L^\E_{l+1}(\omega)$ by their subcategories ${}^0{\Fac}^\E_{l+1}(\omega)$, ${}^0{\C}^\E_{l+1}(\omega)$, and ${}^0{\L}^\E_{l+1}(\omega)$, respectively.
\end{prp}

\begin{cor} \label{cor: cok-L}
	Let $(\tau, \omega)$ be a twisted homothety on an exact category $\A$, $\E = \tau \E$ and $\E_\omega$ fully exact subcategories of $\A$ and $\A/\omega$, respectively, and $l \in \NN$. Suppose that \Cref{asn}.\ref{asn-coker} holds.
	\begin{enumerate}
		\item The functor $\cok$ is the restriction of the functor $L$ from \Cref{ntn: L} along a fully faithful, fully exact, extension-injective functor.
		
		\item Under \Cref{asn}.\ref{asn-ker}, the functors $\cok$ and $L$ agree up to an equivalence of exact categories between their domains. \qed
	\end{enumerate}
\end{cor}

\begin{proof}[Proof of \Cref{prp: comma-equiv}] \
	Due to \Cref{lem: ext-bij,lem: comma-embedding}, it suffices to show that the composite in \ref{prp: comma-equiv-b} is full, fully exact, extension-injective, and, under \Cref{asn}.\ref{asn-ker}, essentially surjective.\\
	To prove \emph{fullness}, consider a morphism $(\tilde f, g) \colon ((j_X^l, q_X^l), \cok_l(X)) \to ((j_Y^l, q_Y^l), \cok_l(Y))$ in $\L^\E_{l+1}(\omega)$ where $X, Y \in \Fac^\E_{l+1}(\omega)$ and $\tilde f = (f^0, f^l)$. Set $U_X := \iota(\cok_l(X))$ and $U_Y := \iota(\cok_l(Y))$. Since $q^l_Y f^l = g^l q_X^l$, the functoriality of pullbacks yields the desired preimage $f \colon X \to Y$:
	\begin{center}
		\begin{tikzcd}[column sep={10mm,between origins}, row sep={12.5mm,between origins}]
			X \ar[rd, dashed, "f"] \ar[dd, two heads, "q_X"] && X^0 \ar[rr, tail] \ar[rd, dashed, "f^0"] \ar[dd, two heads, "q^0_X"' near start] && X^1 \ar[rr, tail] \ar[rd, dashed, "f^1"] \ar[dd, two heads, "q^1_X"' near start] && X^2 \ar[rr, tail] \ar[rd, dashed, "f^2"] \ar[dd, two heads, "q^2_X"' near start] && \cdots \ar[rr, tail] && X^{l-1} \ar[rr, tail] \ar[rd, dashed, "f^{l-1}"] \ar[dd, two heads, "q^{l-1}_X"' near start] && X^l \ar[rd, "f^l"] \ar[dd, two heads, "q^l_X"' near start] & \\
			& Y \ar[dd, two heads, "q_Y"] && Y^0 \ar[rr, tail, crossing over] && Y^1 \ar[rr, tail, crossing over] && Y^2 \ar[rr, tail, crossing over] && \cdots \ar[rr, tail, crossing over] && Y^{l-1} \ar[rr, tail, crossing over] && Y^l \\
			U_X \ar[rd, "g"' near start] && 0 \ar[rr, tail] \ar[rd] && U_X^1 \ar[rr, tail] \ar[rd, "g^1"' near start] && U_X^2 \ar[rr, tail] \ar[rd, "g^2"' near start] && \cdots \ar[rr, tail] && U_X^{l-1} \ar[rr, tail] \ar[rd, "g^{l-1}"' near start] && U_X^l \ar[rd, "g^l"' near start] \\
			& U_Y && 0 \ar[rr, tail] \ar[uu, <<-, crossing over, "q_Y^0"' near end] && U_Y^1 \ar[rr, tail] \ar[uu, <<-, crossing over, "q_Y^1"' near end] && U_Y^2 \ar[rr, tail] \ar[uu, <<-, crossing over, "q_Y^2"' near end] && \cdots \ar[rr, tail] && U_Y^{l-1} \ar[rr, tail] \ar[uu, <<-, crossing over, "q_Y^{l-1}"' near end] && U_Y^l \ar[uu, <<-, crossing over, "q_Y^l"' near end] 
		\end{tikzcd}
	\end{center}
	By uniqueness, the induced morphism $X^0 \to Y^0$ agrees with $f^0$.\\
	The composed functor is \emph{exact}, since $\gamma$ and $\cok_l$ are so, see \Cref{lem: coker-exact}. It also also \emph{reflects exactness}: Given a sequence $(X, \alpha) \xrightarrow{i} (Y, \beta) \xrightarrow{p} (Z, \gamma)$ in $\Fac^\E_{l+1}(\omega)$, suppose that
	\begin{equation} \label{eqn: ses}
		\begin{tikzcd}[sep=large]
			((j_X^l,q_X^l), \cok_l(X)) \ar[r, tail, "{((i^0, i^l), \overline i)}"] & ((j_Y^l,q_Y^l), \cok_l(Y)) \ar[r, two heads, "{((p^0, p^l), \overline p)}"] & ((j_Z^l,q_Z^l), \cok_l(Z)),
		\end{tikzcd}
	\end{equation}
	where $\overline i = \cok_l(i)$ and $\overline p = \cok_l(p)$, is a short exact sequence in $\L^\E_{l+1}(\omega)$. Set $(U_X, \overline \alpha) := \iota(\cok_l(X))$, $(U_Y, \overline \beta) := \iota(\cok_l(Y))$, and $(U_Z, \overline \gamma) := \iota(\cok_l(Z))$. 
	Applying \Cref{prp: Buehler2.12} to \eqref{diag: coker} for $X$, $Y$, and $Z$, yields the commutative diagram
	\begin{center}
		\begin{tikzcd}[row sep={15mm,between origins}, column sep={27.5mm,between origins}]
			X^k \ar[r] \ar[d, tail] & Y^k \ar[r] \ar[d, tail] & Z^k \ar[d, tail] \\
			U_X^k \oplus X^{k+1} \ar[r, tail] \ar[d, two heads] & U_Y^k \oplus Y^{k+1} \ar[r, two heads] \ar[d, two heads] & U_Z^k \oplus Z^{k+1} \ar[d, two heads] \\
			U_X^{k+1} \ar[r, tail] & U_Y^{k+1} \ar[r, two heads] & U_Z^{k+1},
		\end{tikzcd}
	\end{center}

	with short exact columns. For $k=l-1$, the middle and the lower row are also short exact. Hence, due to \Cref{lem: Noether}, applied iteratively for decreasing $k \in \{0, \dots, l-1\}$, the upper row is also short exact. For \emph{extension-injectivity}, suppose that \eqref{eqn: ses} splits, and consider a reverse $\left(((j^0, j^l), \overline j), ((q^0, q^l), \overline q)\right)$. Then, for decreasing $k \in \{0, \dots, l-1\}$, the functoriality of pullbacks yields the dashed morphisms in the commutative diagram
	\begin{equation*}
		\begin{tikzcd}[sep={12mm,between origins}]
			&& Z^{k} \ar[rrr, tail, "\gamma^{k}"] \ar[ddd, two heads, "q_Z^{k}"] \ar[ld, dashed, "j^{k}"'] &&& Z^{k+1} \ar[ddd, two heads, "q_Z^{k+1}"] \ar[ld, tail, "j^{k+1}"' near end] \\
			& Y^{k} \ar[ld, dashed, "q^{k}"'] \ar[rrr, crossing over,tail,"\beta^{k}"] \ar[ddd, two heads,"q_Y^{k}"] &&& Y^{k+1} \ar[ld, two heads, "q^{k+1}"' near end] \\
			X^{k}  \ar[rrr, crossing over, tail, "\alpha^{k}"] \ar[ddd, two heads, "q_X^{k}"] &&& X^{k+1}   \\
			&& U_Z^{k} \ar[rrr, tail, "\overline \gamma^{k}"] \ar[ld, tail, "\overline j^{k}"'] &&& U_Z^{k+1} \ar[ld, tail, "\overline j^{k+1}"] \\
			& U_Y^{k} \ar[ld, two heads, "\overline q^{k}"'] \ar[rrr, tail, "\overline \beta^{k}"] &&& U_Y^{k+1} \ar[ld, two heads, "\overline q^{k+1}"] \ar[uuu, <<-, crossing over, "q_Y^{k+1}"'] \\
			U_X^{k}  \ar[rrr, tail, "\overline \alpha^{k}"] &&& U_X^{k+1}.  \ar[uuu, <<-, crossing over, "q_X^{k+1}"']
		\end{tikzcd}
	\end{equation*}
	By uniqueness, the induced morphisms $Z^k \to Y^k$ and $Y^k \to X^k$ agree with $j^0$ and $q^0$, respectively. By \Cref{lem: restrict-splitting}, $(j^k, q^k)$ is a reverse of $(i^k,p^k)$, for all $k \in \{0,\dots,l\}$, and hence $(j,p)$ a reverse  of $(i,p)$.\\
	Suppose that \Cref{asn}.\ref{asn-ker} holds.
	For \emph{essential surjectivity}, consider an arbitrary object $((\kappa, \sigma), V) \in \L^\E_{l+1}(\omega)$. By successive pullbacks from left to right, we obtain the front layer of the following commutative diagram, consisting of the given and dashed arrows and whose columns are short exact, see \Cref{prp: Buehler2.12}.\ref{prp: Buehler2.12-pull}:
	\begin{align} 
		\begin{tikzcd}[sep={11.5mm,between origins}, ampersand replacement=\&]
			\& \nu^l(X^0) \ar[ld, <-, dotted, "\nu^l(f^0)"'] \ar[dd, tail, "j_X"] \&\& X^0 \ar[ld, <-, dotted, "f^0"' near end] \ar[rr, equal] \ar[dd, equal] \&\& X^0 \ar[ld,<-, dotted, "f^0"' near end] \ar[rr, equal] \ar[dd, tail] \&\& X^0 \ar[ld, <-, dotted, "f^0"' near end] \ar[rr, equal] \ar[dd, tail] \&\& \cdots \ar[rr, equal] \&\& X^0 \ar[ld, <-, dotted, "f^0"' near end] \ar[rr, equal] \ar[dd, tail] \&\& X^0 \ar[ld, <-, dotted, "f^0"' near end] \ar[dd, tail, "j_X^l"] \ar[dd, tail] \\
			\nu^l(Y^0) \ar[dd, tail, dashed, "\tilde j"'] \&\& Y^0 \ar[rr, equal, crossing over] \&\& Y^0 \ar[rr, equal, crossing over] \&\& Y^0 \ar[rr, equal, crossing over] \&\& \cdots \ar[rr, equal, crossing over] \&\& Y^0 \ar[rr, equal, crossing over] \&\& Y^0 \\
			\& X \ar[ld, equal] \ar[dd, two heads, "q_X"] \&\& X^{0} \ar[ld, equal] \ar[rr, tail, "\alpha^{0}" near end] \ar[dd, two heads] \&\& X^{1} \ar[ld, equal] \ar[rr, tail, "\alpha^{1}" near end] \ar[dd, two heads] \&\& X^{2} \ar[ld, equal] \ar[rr, tail, "\alpha^{2}"] \ar[dd, two heads] \&\& \cdots \ar[rr, tail, "\alpha^{l-2}" near start] \&\& X^{l-1} \ar[ld, equal] \ar[rr, tail, "\alpha^{l-1}" near end] \ar[dd, two heads] \&\& \ar[ld, equal] X^l \ar[dd, two heads, "q_X^l"]\\
			X \ar[dd, two heads, dashed, "\tilde q"'] \&\& X^{0} \ar[rr, tail, dashed, crossing over, "\alpha^{0}" near end] \ar[uu, <-<, crossing over, dashed] \&\& X^{1} \ar[rr, tail, dashed, crossing over,"\alpha^{1}" near end] \ar[uu, <-<, crossing over, dashed] \&\& X^{2} \ar[rr, tail, dashed, crossing over,"\alpha^{2}" near end] \ar[uu, <-<, crossing over, dashed] \&\& \cdots \ar[rr, tail, dashed, crossing over,"\alpha^{l-2}"] \&\& X^{l-1} \ar[rr, tail, dashed, crossing over,"\alpha^{l-1}" near end] \ar[uu, <-<, crossing over, dashed] \&\& Y^l \ar[uu, <-<, crossing over, "\kappa"' near end] \\
			\& \iota \cok(X) \&\& 0 \ar[rr, tail] \&\& U_X^{1} \ar[rr, tail] \&\& U_X^{2} \ar[rr, tail] \&\& \cdots \ar[rr, tail] \&\& U_X^{l-1} \ar[rr, tail] \&\& U_X^l\\
			\iota V \ar[ru, dotted, "\iota g"'] \&\& 0 \ar[uu, <<-, crossing over, dashed] \ar[rr, tail] \ar[ru, dotted] \&\& V^{1} \ar[uu, <<-, crossing over, dashed] \ar[rr, tail] \ar[ru, dotted, "g^1"'] \&\& V^{2} \ar[uu, <<-, crossing over, dashed] \ar[rr, tail] \ar[ru, dotted, "g^2"'] \&\& \cdots \ar[rr, tail] \&\& V^{l-1} \ar[uu, <<-, crossing over, dashed] \ar[rr, tail] \ar[ru, dotted, "g^{l-1}"'] \&\& V^l \ar[uu, <<-, crossing over, "\sigma"' near end] \ar[ru, dotted, "g^l"'] 
		\end{tikzcd}
	\end{align}
	The back layer of the diagram is \eqref{diag: coker} for $X = (X, \alpha) \in \mMor_{l}(\A)$, where $X^l := Y^l$. Since $X^l \in \E$ and $\cok(\monic{X^k}{X^l}) \cong \cok(\monic{V^k}{V^l}) \in \E_\omega$, see \Cref{prp: Buehler2.12}.\ref{prp: Buehler2.12-push}, we have $X^k \in \E$ by \Cref{asn}.\ref{asn-ker} for $k \in \{0, \dots, l-1\}$. Hence, $X \in \Fac^\E_{l+1}(\omega)$ due to \Cref{rmk: Fac-cok-A/omega} applied to the short exact sequence
	\[\begin{tikzcd}[sep={30mm,between origins}] X^0 \ar[r, tail, "j_X^l \, = \, \alpha^{l-1} \cdots \alpha^0"] & X^l \ar[r, two heads, "\sigma"] & V^l, \end{tikzcd}\]
	see \Cref{prp: Buehler2.12}.\ref{prp: Buehler2.12-pull}. It remains to construct the dotted isomorphisms: Due to the sequence, $j_X^l$ is a kernel of $\sigma$. This yields an isomorphism $\tilde f=(f^0,\id_{X^l}) \colon \kappa \cong j_X^l$, where $f^0 \colon Y^0 \cong X^0$. It follows that $(\nu^l(f^0), \id_X) \colon \tilde j \cong j_X$ is an isomorphism. Since $(\tilde j, \tilde q)$ and $(j_X, q_X)$ are short exact sequences in $\mMor_l(\A)$, there is then an isomorphism $(\id_X, \iota g) \colon \tilde q \cong q_X$, where $g \colon V \cong \cok_l(X)$. Hence, we obtain the desired isomorphism $(\tilde f,g)\colon ((\kappa, \sigma), V) \cong ((j^l_X, q^l_X), \cok_l(X))$.\\
	All of the previous arguments are unaffected by imposing the condition $X^l \in \Proj(\A)$ on all occurring objects. The claims on the restricted functors follow.
\end{proof}

\section{Categorical Eisenbud--Yoshino Theorem}

In this section, we formulate and prove our categorical version of the matrix factorization theorem due to Eisenbud and Yoshino. 

\begin{prp} \label{prp: cok-restricted}
	Let $(\tau, \omega)$ be a twisted homothety on an exact category $\A$ and $l \in \NN$. Consider fully exact subcategories $\E = \tau \E$ and $\E_\omega$ of $\A$ and $\A/\omega$, respectively. Suppose that \Cref{asn}.\ref{asn-coker} holds. Then the restricted functor $\cok \colon {}^0{\Fac}^\E_{l+1}(\omega) \to \mMor_{l-1}(\E_\omega)$ is full and extension-injective. Under \Cref{asn}.\ref{asn-ker}, it is extension-bijective. If, in addition, $\E$ has enough $\A$-projectives and \Cref{asn}.\ref{asn-epic} holds, it is also essentially surjective.
\end{prp}

\begin{proof}
	By \Cref{cor: cok-L}, it suffices to show the claims for the restricted functor $L \colon {}^0{\L}^\E_{l+1}(\omega) \to \mMor_{l-1}(\E_\omega)$ from \Cref{ntn: L}. To prove \textit{fullness}, consider a morphism $g \colon U \to V$ in $\mMor_{l-1}(\E_\omega)$ for $((\iota,\rho), U), ((\kappa,\sigma), V) \in {}^0\L^\E_{l+1}(\omega)$, where $\iota \colon \monic{X^0}{X^l}$ and $\kappa \colon \monic{Y^0}{Y^l}$. Using that $X^l \in \Proj(\A)$, we obtain a commutative diagram
	\begin{equation*}
		\begin{tikzcd}[sep={15mm,between origins}]
			X^0 \ar[r, tail, "\iota"] \ar[d, "f^0", dashed] & X^l \ar[r, two heads, "\rho"] \ar[d, dashed, "f^l"] & U^l \ar[d, "g^l"] \\
			Y^0 \ar[r, tail, "\kappa"] & Y^l \ar[r, two heads, "\sigma"] & V^l
		\end{tikzcd}
	\end{equation*}
	in $\A$, which yields the desired preimage $(\tilde  f,g) = ((f^0, f^l), g)$ of $g$. For \emph{extension-injectivity}, consider a short exact sequence
	\begin{equation} \label{eqn: ses-2}
		\begin{tikzcd}[sep=large]
			((\iota,\rho), U) \ar[r, tail, "{((i^0, i^l), \overline i)}"] & ((\kappa,\sigma), V) \ar[r, two heads, "{((p^0, p^l), \overline p)}"] & ((\lambda,\tau), W),
		\end{tikzcd}
	\end{equation}
	in ${}^0{\L}^\E_{l+1}(\omega)$, where $\iota \colon \monic{X^0}{X^l}$, $\kappa \colon \monic{Y^0}{Y^l}$, and $\lambda \colon \monic{Z^0}{Z^l}$. Suppose that \begin{tikzcd}[cramped, sep=small] U \ar[r, tail, "\overline i"] & V \ar[r, two heads, "\overline p"] & W \end{tikzcd} splits, and pick a reverse $(\overline j, \overline q)$. Apply \Cref{lem: lift-splitting} to
	\begin{equation*} 
		\begin{tikzcd}[sep={15mm,between origins}]
			X^l \ar[r, tail, "i^l"] \ar[d, two heads, "\rho"] 
			& Y^l \ar[r, two heads, "p^l"] \ar[d, two heads, "\sigma"] 
			& Z^l \ar[d, two heads, "\tau"] \\
			U^l  \ar[r, tail, "\overline i^l"] 
			& V^l \ar[r, two heads, "\overline p^l"] 
			& W^l.
		\end{tikzcd}
	\end{equation*}
	This yields a reverse $(j^l, q^l)$ of $(i^l, p^l)$ such that the following diagram commutes, where the dotted morphisms are due to \Cref{lem: Noether}:
	\begin{equation*} 
		\begin{tikzcd}[sep={15mm,between origins}]
			X^0 \ar[d, tail, "\iota"] 
			& Y^0 \ar[l, two heads, dotted, "q^0"'] \ar[d, tail, "\kappa"] 
			& Z^l \ar[d, tail, "\lambda"] \ar[l, tail, dotted, "j^0"'] \\
			X^l \ar[d, two heads, "\rho"] 
			& Y^l \ar[l, two heads, dashed, "q^l"'] \ar[d, two heads, "\sigma"] 
			& Z^l \ar[d, two heads, "\tau"] \ar[l, tail, dashed, "j^l"'] \\
			U^l  
			& V^l \ar[l, two heads, "\overline q^l"'] 
			& W^l \ar[l, tail, "\overline j^l"'] 
		\end{tikzcd}
	\end{equation*}
	By \Cref{lem: restrict-splitting}, $(j^0, q^0)$ is a reverse of $(i^0,p^0)$. Combining $((j^0,j^l),(q^0,q^l))$ and $(\overline j, \overline q)$, then yields a reverse of \eqref{eqn: ses-2}, due to the termwise exact structure.\\
	Now suppose that \Cref{asn}.\ref{asn-ker} holds. 
	For \emph{extension-surjectivity}, apply the Horseshoe lemma, see {\cite[Thm.~12.8]{Buh10}}, to a short exact sequence $\begin{tikzcd}[cramped, sep=small]U \ar[r, tail, "\overline i"] & V \ar[r, two heads, "\overline p"] & W \end{tikzcd}$ in $\mMor_{l-1}(\E_\omega)$, where $((\iota,\rho), U), ((\lambda,\tau), W) \in {}^0{\L}^\E_{l+1}(\omega)$, $\iota \colon \monic{X^0}{X^l}$, and $\lambda \colon \monic{Z^0}{Z^l}$: We obtain an admissible epic $\sigma \colon \epic{Y^l}{V^l}$ in $\A$, where $Y^l := X^l \oplus Z^l  \in \Proj(\A) \cap \E$. Denote its kernel by $\kappa \colon \monic{Y^0}{Y^l}$. The short exact sequence $(\kappa, \sigma)$ fits into a commutative diagram
	\begin{equation*} 
		\begin{tikzcd}[sep={15mm,between origins}]
			X^0 \ar[r, tail, dotted, "i^0"] \ar[d, tail, "\iota"] & Y^0 \ar[r, two heads, dotted, "p^0"] \ar[d, tail, dashed, "\kappa"] & Z^0 \ar[d, tail, "\lambda"] \\
			X^l \ar[r, tail, dashed, "i^l"] \ar[d, two heads, "\rho"] 
			& Y^l \ar[r, two heads, dashed, "p^l"] \ar[d, two heads, dashed, "\sigma"] 
			& Z^l \ar[d, two heads, "\tau"] \\
			U^l  \ar[r, tail, "\overline i^l"] 
			& V^l \ar[r, two heads, "\overline p^l"] 
			& W^l,
		\end{tikzcd}
	\end{equation*}
	with $i^l = \begin{pmatrix} 1 & 0 \end{pmatrix}$ and $p^l = \begin{pmatrix} 0 \\ 1 \end{pmatrix}$, where the dotted short exact sequence is due to \Cref{lem: Noether}. Note that $Y^0 \in \E$ by \Cref{asn}.\ref{asn-ker}.  Combining $\left((i^0,i^l), (p^0,p^l)\right)$ with $(\overline i, \overline p)$, then yields a preimage of $(\overline i, \overline p)$, due to the termwise exact structure.
	For \emph{essential surjectivity}, suppose, in addition, that $\E$ has enough $\A$-projectives and that \Cref{asn}.\ref{asn-epic} holds. Given $U \in \mMor_{l-1}(\E_\omega)$, this yields an admissible epic $\rho \colon \epic{X^l}{U^l}$ in $\A$ with $X^l \in \Proj(\A) \cap \E$. Due to \Cref{asn}.\ref{asn-ker}, its kernel $\iota$ is a morphism in $\E$. Hence, $((\iota, \rho), U)$ is the desired preimage of $U$.
\end{proof}

\begin{ntn}
	Let $(\tau, \omega)$ be a twisted homothety on an exact category $\A$ and $l \in \NN$. Consider a fully exact subcategory $\E = \tau \E$ of $\A$. We use shorthand notations for the following two quotient categories:
	\[{\widetilde \Fac}^\E_{l+1}(\omega) := \Fac^\E_{l+1}(\omega)/\nu^l(\E)\]
	and
	\[{}^0{\widetilde \Fac}^\E_{l+1}(\omega) := {}^0{\Fac}^\E_{l+1}(\omega)/\nu^l(\Proj(\A) \cap \E).\]	
\end{ntn}

\begin{lem} \label{lem: quot}
	Let $(\tau, \omega)$ be a twisted homothety on an exact category $\A$ and $l \in \NN$. Consider a fully exact subcategory $\E = \tau \E$ of $\A$. If $\E$ has enough $\A$-projectives, then the canonical functor \[{}^0{\widetilde \Fac}^\E_{l+1}(\omega) \to {\widetilde \Fac}^\E_{l+1}(\omega)\] is fully faithful.
\end{lem}

\begin{proof}
	Fullness is clear. To prove faithfulness, consider an object $X \in {}^0{\Fac}^\E_{l+1}(\omega)$ and a morphism $g \colon X \to \nu^l(A)$ in $\Fac^\E_{l+1}(\omega)$ for some $A \in \E$. By assumption, there is an admissible epic $p\colon \epic{P}{A}$ in $\E$ with $P \in \Proj(\A)$. Use that $X^l \in \Proj(\A)$ to obtain a lift $\hat g^l \colon X^l \to P$ of $g^l$ along $p$. Under the adjunction from \Cref{lem: nu-adjunction}.\ref{lem: nu-adjunction-right}, it corresponds to a lift $\hat g \colon X \to \nu^l(P)$ of $g$ along $\nu^l(p)$:
	\begin{equation*}
		\begin{tikzcd}[sep={12.5mm,between origins}]
			X \ar[rd, dashed, "\hat g"] \ar[dd, "g"'] && X^{0} \ar[rr, tail, "\alpha^{0}"] \ar[rd, dashed, "\hat g^{0}"] \ar[dd, "g^{0}"' near start] && X^{1} \ar[rr, tail, "\alpha^{1}"] \ar[rd, dashed, "\hat g^{1}"] \ar[dd, "g^{1}"' near start] && \cdots \ar[rr, tail, "\alpha^{l-2}"] && X^{l-1} \ar[rr, tail, "\alpha^{l-1}"] \ar[rd, dashed, "\hat g^{l-1}"] \ar[dd, "g^{l-1}"' near start] && X^{l} \ar[dd, "g^{l}"' near start] \ar[rd, dashed, "\hat g^{l}"] \\
			& \nu^l(P) \ar[ld, "\nu^l(p)"] && P \ar[rr, equal, crossing over] \ar[ld, two heads, "p"] && P \ar[rr, equal, crossing over] \ar[ld, two heads, "p"] && \cdots \ar[rr, equal, crossing over] && P \ar[rr, equal, crossing over] \ar[ld, two heads, "p"] && P \ar[ld, two heads, "p"] \\
			\nu^l(A) && A \ar[rr, equal] && A \ar[rr, equal] && \cdots \ar[rr, equal] && A \ar[rr, equal] && A 
		\end{tikzcd}
	\end{equation*}
	It follows that a morphism factors through an object of $\nu^l(\Proj(\E))$ if it factors through an object of $\nu^l(\E)$.
\end{proof}

\begin{prp} \label{prp: coker-induced}
	Let $(\tau, \omega)$ be a twisted homothety on an exact category $\A$ and $l \in \NN$. Consider fully exact subcategories $\E = \tau \E$ and $\E_\omega$ of $\A$ and $\A/\omega$, respectively. Suppose that \Cref{asn}.\ref{asn-coker} holds. Then the cokernel induces a faithful functor
	\[{\widetilde \Fac}^\E_{l+1}(\omega) \xlongrightarrow{\cok} \mMor_{l-1}(\E_\omega).\]
	In particular, its restriction to ${}^0{\widetilde \Fac}^\E_{l+1}(\omega)$ is faithful if $\E$ has enough $\A$-projectives.
\end{prp}

\begin{proof}
	The subcategory $\N := \{((\id_A,0), 0) \mid A \in \E\}$ of $\L^\E_{l+1}(\omega)$ is the full image of $\nu^l(\E)$ under the fully faithful functor $\Fac^\E_{l+1}(\omega) \hookrightarrow \L^\E_{l+1}(\omega)$ from \Cref{prp: comma-equiv}. Therefore, it induces a fully faithful functor ${\widetilde \Fac}^\E_{l+1}(\omega) \hookrightarrow \L^\E_{l+1}(\omega)/\N$, see \Cref{lem: equivalence-quotient}. The functor $L$ from \Cref{ntn: L} sends the objects of $\N$ to zero, and thus induces a functor $\L^\E_{l+1}(\omega)/\N \to \mMor_{l-1}(\E_\omega)$.
	%The same arguments work for the subcategory ${}^0{\N} := \{((\id_A,0), 0) \mid A \in \Proj(\A)\cap\E\}$ and yields functors $\overline L \colon {}^0{\L}^\E_{l+1}(\omega)/{}^0{\N} \to \mMor_{l-1}(\E_\omega)$ and ${}^0{\widetilde \Fac}^\E_{l+1}(\omega) \hookrightarrow \L^\E_{l+1}(\omega)/\N$.
	These functors fit into a commutative diagram
	\begin{equation*}
		%	\begin{tikzcd}[row sep={17.5mm,between origins}, column sep={17.5mm,between origins}]
			%		& \N \ar[rr, hook] && \L^\E_{l+1}(\omega) \ar[rr]  && \L^\E_{l+1}(\omega)/\N \ar[rr, "\overline L"] && \mMor_{l-1}(\E_\omega) \\
			%		{}^0{\N} \ar[ru, hook] && {}^0{\L}^\E_{l+1}(\omega) \ar[ru, hook] && {}^0{\L}^\E_{l+1}(\omega)/{}^0{\N}  \ar[ru, hook] && \mMor_{l-1}(\E_\omega) \ar[ru, equal] \\
			%		& \nu^l(\Proj(\E)) \ar[rr, hook] \ar[uu, "\simeq" near end] && \Fac^\E_{l+1}(\omega) \ar[rr] \ar[uu, hook] && {\widetilde \Fac}^\E_{l+1}(\omega) \ar[rr, "\cok" near start] \ar[uu, hook] && \mMor_{l-1}(\E_\omega). \ar[uu, equal]  \\
			%		\nu^l(\Proj(\A)\cap \E) \ar[rr, hook] \ar[ru, hook] \ar[uu, crossing over, "\simeq"] && {}^0{\Fac}^\E_{l+1}(\omega) \ar[ru, hook] \ar[rr] \ar[uu, hook, crossing over] && {}^0{\widetilde \Fac}^\E_{l+1}(\omega) \ar[ru, hook] \ar[rr, "\cok"] \ar[uu, hook, crossing over] && \mMor_{l-1}(\E_\omega), \ar[ru, equal] \ar[uu, equal, crossing over]
			%		\arrow[from=2-1, to=2-3, crossing over, hook]
			%		\arrow[from=2-3, to=2-5, crossing over]
			%		\arrow[from=2-5, to=2-7, crossing over, "\overline L" near end]
			%	\end{tikzcd}
		\begin{tikzcd}[row sep={17.5mm,between origins}, column sep={35mm,between origins}]
			\N \ar[r, hook] & \L^\E_{l+1}(\omega) \ar[r] & \L^\E_{l+1}(\omega)/\N \ar[r, "L"]& \mMor_{l-1}(\E_\omega)\\
			\nu^l(\E) \ar[r, hook] \ar[u, "\simeq"] & \Fac^\E_{l+1}(\omega) \ar[u, hook] \ar[r] & {\widetilde \Fac}^\E_{l+1}(\omega) \ar[u, hook] \ar[r, "\cok"] & \mMor_{l-1}(\E_\omega). \ar[u, equal]
		\end{tikzcd}
	\end{equation*}
	Due to the right square, it suffices to prove the claim for $L$ instead of $\cok$.
	Consider a morphism $(\tilde f,g) \colon ((\iota,\rho), U) \to ((\kappa,\sigma), V)$ in $\L^\E_{l+1}(\omega)$, where $\tilde f = (f^0, f^l)$, $\iota \colon \monic{X^0}{X^l}$, and $\kappa \colon \monic{Y^0}{Y^l}$. Supposing that $g=0$, the morphism $f^l$ factors through $\kappa$ as follows:
	\begin{equation*}
		\begin{tikzcd}[sep={17.5mm,between origins}]
			X^0 \ar[r, tail, "\iota"] \ar[d, "f^0"] & X^l \ar[r, two heads, "\rho"] \ar[d, "f^l"] \ar[ld, dashed, "\hat f^l"] & U^l \ar[d, "g^l \, = \, 0"] \\
			Y^0 \ar[r, tail, "\kappa"] & Y^l \ar[r, two heads, "\sigma"] & V^l
		\end{tikzcd}
	\end{equation*}
	Then $\tilde f=(f^0,f^l)=(\id_{Y^0}, \kappa) \circ (f^0, \hat f^l)$ factors though $\id_{Y^0}$:
	\begin{equation*}
		\begin{tikzcd}[sep={17.5mm,between origins}]
			X^0 \ar[r, tail, "\iota"] \ar[d, "f^0"] & X^l  \ar[d, "\hat f^l"] \ar[dd, bend left=15mm, "f^l"] \\
			Y^0 \ar[r, equal] \ar[d, equal] & Y^0 \ar[d, tail, "\kappa"] \\
			Y^0 \ar[r, tail, "\kappa"] & Y^l,
		\end{tikzcd}
	\end{equation*}
	It follows that $(\tilde f,g)$ factors through $((\id_{Y^0},0), 0) \in \N$. The particular claim is due to \Cref{lem: quot}.\qedhere
\end{proof}

Combining \Cref{prp: homomorphism-thm-reversed,prp: cok-restricted,prp: coker-induced} yields

\begin{cor} \label{cor: Fac-mMor-exact}
	Let $(\tau, \omega)$ be a twisted homothety on an exact category $\A$ and $l \in \NN$. Consider fully exact subcategories $\E = \tau \E$ and $\E_\omega$ of $\A$ and $\A/\omega$, respectively. Suppose that $\E$ has enough $\A$-projectives and that \Cref{asn} holds. Then the cokernel functor induces an equivalence of exact categories
	\[{}^0{\widetilde \Fac}^\E_{l+1}(\omega) \xlongrightarrow[\simeq]{\cok} \mMor_{l-1}(\E_\omega),\]
	where ${}^0{\widetilde \Fac}^\E_{l+1}(\omega)$ carries the image exact structure induced by the functor  ${}^0{\Fac}^\E_{l+1}(\omega) \to {}^0{\widetilde \Fac}^\E_{l+1}(\omega)$, which is then extension-bijective. \qed
\end{cor}

\begin{thm} \label{thm: coker-equiv}
	Let $(\tau, \omega)$ be a twisted homothety on an exact weakly idempotent complete category $\A$ and $l \in \NN$. Consider fully exact Frobenius subcategories $\E = \tau \E$ and $\E_\omega$ of $\A$ and $\A/\omega$, respectively. Suppose that $\E$ has enough $\A$-projectives, that $(\tau, \omega)$ is regular on $\E$, and that \Cref{asn} holds.
	Then the cokernel functor induces a triangle equivalence
	\[{}^0{\ul \Fac}^\E_{l+1}(\omega) = {}^0\Fac^\E_{l+1}(\omega) / \langle \nu^k(P) \, \vert \, k \in \{0, \dots, l\}, P \in \Proj(\E) \rangle \xrightarrow[\simeq]{\cok} \ul \mMor_{l-1}(\E_\omega).\]
\end{thm}

\begin{proof}
	Both stable categories are triangulated and the equality holds due to \Cref{thm: stable-Frobenius,thm: mMor} and \Cref{prp: Fact-enough}. By \Cref{cor: Fac-mMor-exact}, there is an equivalence
	\begin{gather*}
		{}^0{\widetilde\Fac}^\E_{l+1}(\omega)  \xrightarrow[\simeq]{\cok} \mMor_{l-1}(\E_\omega),
	\end{gather*}
	which sends $\cI := \langle \nu^k(P) \, \vert \, k \in \{0, \dots, l-1\}, P \in \Proj(\E) \rangle$ onto $\mathcal J := \langle \mu_k(\ol P) \, \vert \, k \in \{1, \dots, l\}, P \in \Proj(\E) \rangle$, see \Cref{ntn: overline-omega} and \Cref{rmk: tau-restricted}. Due to \Cref{lem: Proj-A/omega} and \Cref{asn}.\ref{asn-coker}, $\ol P \in \Proj(\A/\omega) \cap \E_\omega \subseteq \Proj(\E_\omega)$ for any $P \in \Proj(\E)= \E \cap \Proj(\A)$, see \Cref{lem: sub-Frobenius}.\ref{lem: sub-Frobenius-a}. Hence, $\mathcal J \subseteq \Proj(\mMor_{l-1}(\E_\omega))$ by \Cref{thm: mMor}.\ref{thm: mMor-Proj}, and $\mMor_{l-1}(\E_\omega) / \mathcal J = \ul \mMor_{l-1}(\E_\omega)$ due to \Cref{con: mMor-enough} and \Cref{lem: epic-overline-P}. By \Cref{prp: Fact-enough} and \Cref{lem: iterated-quotient}, ${}^0{\widetilde\Fac}^\E_{l+1}(\omega)/\cI \cong {}^0{\ul\Fac}^\E_{l+1}(\omega)$. By \Cref{lem: equivalence-quotient}, the induced functor
	\[{}^0{\ul\Fac}^\E_{l+1}(\omega) \xrightarrow[\simeq]{\cok} \ul \mMor_{l-1}(\E_\omega)\]
	is an equivalence and triangulated due to \Cref{lem: coker-E_omega} and \Cref{prp: stable-functor-triang}. Its quasi-inverse is automatically a triangle functor, see {\cite[Prop.~1.4]{BK89}} for a more general statement.
\end{proof}

\begin{lem} \label{lem: epic-overline-P}
	Let $(\tau, \omega)$ be a twisted homothety on an exact weakly idempotent complete category $\A$, $\E = \tau \E$ and $\E_\omega$ fully exact subcategories of $\A$ and $\A/\omega$, respectively. Suppose that $(\tau, \omega)$ is regular on $\E$, that $\E$ has enough projectives and that \Cref{asn} holds. Then each $X \in \E_\omega$ admits an admissible epic $\epic{\ol P}{X}$ in $\E_\omega$ with $P \in \Proj(\E)$.
\end{lem}

\begin{proof}
	Since \Cref{asn}.\ref{asn-epic} holds and $\E$ has enough projectives, for each $X \in \E_\omega$, there is an admissible epic $p \colon \epic{P}{X}$ in $\A$ with $P \in \Proj(\E)$. With $i\colon \monic{Y}{P}$ denoting its kernel, $Y$ lies in $\E$ due to \Cref{asn}.\ref{asn-ker}. By \Cref{lem: A/omega-coker}, $\omega_{\tau^{-1}P}=i j$ for a morphism $j \colon \monic{\tau^{-1} P}{Y}$, which is a monic in $\A$, see \Cref{prp: WIC}.\ref{prp: WIC-comp-monic}. With $Z$ denoting its cokernel, \Cref{lem: Noether} then yields a commutative diagram
	\begin{equation} \label{eqn: epic-overline-P}
		\begin{tikzcd}[sep={17.5mm,between origins}]
			\tau^{-1} P \ar[r, tail, "j"] \ar[d, equal] & Y \ar[r, two heads] \ar[d, tail, "i"] & Z \ar[d, tail, dashed]\\
			\tau^{-1} P \ar[r, tail, "\omega_{\tau^{-1} P}"] \ar[d, two heads] & P \ar[r, two heads, "\ol \omega_{P}"] \ar[d, two heads, "p"] & \ol P \ar[d, two heads, dashed, "\ol p"]\\
			0 \ar[r, tail] & X \ar[r, equal] & X
		\end{tikzcd}
	\end{equation}
	in $\A$ with short exact rows and columns. By \Cref{asn}.\ref{asn-coker}, we have $\ol P \in \E_\omega$, see \Cref{ntn: overline-omega}. Since $i$ is monic, $i\omega_{\tau^{-1}Y}=\omega_{\tau^{-1}P} \tau^{-1}(i)=ij \tau^{-1}(i)$ implies that $\omega_{\tau^{-1}Y} = j \tau^{-1}(i)$ factors through $j$, and hence $Z \in \A/\omega$ due to \Cref{lem: A/omega-coker}. By \Cref{asn}.\ref{asn-coker} applied to the upper row of \eqref{eqn: epic-overline-P}, this means that $Z \in \E_\omega$, which makes $\ol p$ the desired epic.
\end{proof}

We include the following statements for lack of reference:

\begin{lem} \label{lem: iterated-quotient}
	Let $\A$ be a category. Given any two subcategories $\cS$ and $\T$ of $\A$, closed under biproducts, there is a commutative diagram
	\begin{equation*} 
		\begin{tikzcd}[row sep={17.5mm,between origins}, column sep={35mm,between origins}]
			\A \ar[r]\ar[d] & \A/\cS \ar[d] \\
			\A/ \langle \cS \cup \T \rangle \ar[r, dashed, "\cong"] & (\A/\cS)/\T,
		\end{tikzcd}
	\end{equation*}
	where the solid arrows denote the canonical quotient functors. \qedhere
\end{lem}

\begin{proof} Let $P$ denote the composition $\A \to \A/\cS \to (\A/\cS) / \T$. Since $\cS \subseteq \langle \cS \cup \T \rangle$, there is a unique functor $Q \colon \A/\cS \to \A/ \langle \cS \cup \T \rangle$, compatible with the respective quotient functors:
	\begin{equation*}
		\begin{tikzcd}[sep={15mm,between origins}]
			& \A \ar[rd] \ar[ld] \\
			\A/\cS \ar[rr, dashed, "Q"] && \A/\langle \cS \cup \T \rangle
		\end{tikzcd}
	\end{equation*}
	Now, consider a morphism $f \colon A \to B$ in $\A$ which is zero in $\A/ \langle \cS \cup \T \rangle$. By \Cref{rmk: factor-through-biproduct}, there are objects $S \in \cS$, $T \in \T$ and morphisms $r \colon S \to B$, $s \colon A \to S$, $t \colon T \to B$, and $u \colon A \to T$ such that $f=rs+tu$. Thus, the morphism $\ol f = \ol{rs} + \ol{tu} = \ol{tu}$ in $\A/\cS$ factors through $T$ and $f$ is zero in $(\A/\cS)/\T$. Hence, $P$ induces a unique functor $\overline P \colon \A/ \langle \cS \cup \T \rangle \to (\A/\cS) / \T$.\\
	%The left trapezoid is \eqref{diag: iterated-quotient}. The commutativity of the right one follows by precomposition with the full functor $\A \to \A/\cS$.\\ 
	Conversely, suppose that the class $\ol f$ of $f$ in $\A/ \cS$ is zero in $(\A/\cS)/\T$. This means that $\ol f = \ol{tu}$ for an object $T \in \T$ and morphisms $\ol t \colon T \to B$ and $\ol u \colon A \to T$. Then $f-tu$ is zero in $\A/\cS$, that is, $ f - tu= rs$ for an object $S \in \cS$ and morphisms $r \colon S \to B$ and $s \colon A \to S$. Thus, $ f = rs+tu$ is zero in $\A/\langle \cS \cup \T \rangle$ due to \Cref{rmk: factor-through-biproduct}. Hence, $Q$ induces a unique functor $\overline Q \colon (\A/\cS)/\T \to \A/ \langle \cS \cup \T \rangle$.
	%which fits into the square $(Y)$ of the diagram
	%The square $(X)$ is the outer rectangle of \eqref{diag: overline P}. By uniqueness, the concatenations $(XY)$ and $(YZ)$ are the left and the right trapezoid of \eqref{diag: overline P}, respectively. Hence, the compositions $\overline Q \overline P$ and $\overline P \overline Q$ are both the identity, and the claim follows.
	All of these functors fit into a diagram
	\begin{equation*}
		\begin{tikzcd}[sep={20mm,between origins}]
			\A \ar[rd] \ar[rrd, "P"] \ar[rdd] & & \\
			& \A/\cS \ar[d, "Q"', near start] \ar[r] & \left(\A/\cS\right)/\T \ar[ld, "\overline Q", dashed, shift left=1mm] \\
			& \A / \langle \cS \cup \T \rangle. \ar[ru, "\overline P", dashed, shift left=1mm]
		\end{tikzcd}
	\end{equation*}
	By precomposition with the full functors $\A \to \A/\langle \cS \cup \T \rangle$ and $P$, we obtain that $\overline {QP}$ and $\overline{PQ}$ are the identity, respectively, and the claim follows.
\end{proof}

\begin{lem} \label{lem: equivalence-quotient}
	Let $F \colon \A \hookrightarrow \B$ be a fully faithful functor, $\cS$ a subcategory of $\A$, closed under biproducts, and $\T$ its full image in $\B$. Then the induced functor $\ol F \colon \A/\cS \to \B/\T$ between quotient categories remains fully faithful, and is an equivalence if $F$ is so.
\end{lem}

\begin{proof}
	Fullness carries over immediately. For faithfulness, consider a morphism $f \colon A \to B$ in $\A$ such that $F(f)$ factors through an object $F(S) \in \T$, where $S \in \cS$:
	\begin{equation*}
		\begin{tikzcd}[sep={15mm,between origins}]
			F(A) \ar[rr, "F(f)"] \ar[rd, "u"] && F(B) \\
			& F(S) \ar[ru, "t"]
		\end{tikzcd}
	\end{equation*}
	Since $F$ is full, there are morphisms $r \colon S \to B$ and $s \colon A \to S$ such that $t=F(r)$ and $u=F(s)$. Then $F(f)=F(r)F(s)=F(rs)$ implies $f=rs$ since $F$ is faithful. The claim on essential surjectivity is obvious.
\end{proof}

%%%%%%%%%%%%%%%%%%%%%%%%%%%%%%%%%%%%%%%%%%%%%%%%%%%%%%%%%%%%%%%%%%%%%%%%%%%%%%%%

\printbibliography

%%%%%%%%%%%%%%%%%%%%%%%%%%%%%%%%%%%%%%%%%%%%%%%%%%%%%%%%%%%%%%%%%%%%%%%%%%%%%%%%
\end{document}